\documentclass[12pt]{article}
\usepackage{tcolorbox}
\tcbset{arc=0mm,colframe=white,colback=black!7.5!white}
\usepackage{graphicx}
\usepackage{amsmath,amsthm,amssymb,enumerate}
\usepackage{mathtools}
\usepackage{euscript,mathrsfs}
\usepackage{dsfont}
\usepackage{xcolor}
\usepackage{empheq}
\usepackage{float}
\usepackage{subfig}
\usepackage{caption}
\usepackage[left=2cm,right=2cm,top=3.5cm,bottom=3.5cm]{geometry}
\mathtoolsset{showonlyrefs=false}
\usepackage{color}
\catcode`\@=11 \@addtoreset{equation}{section}

\catcode`\@=12

\def\cred{\color{black}}

\newtheorem{Theorem}{Theorem}[section]
\newtheorem{Proposition}[Theorem]{Proposition}
\newtheorem{Assumptions}[Theorem]{Assumptions}
\newtheorem{Lemma}[Theorem]{Lemma}
\newtheorem{Corollary}[Theorem]{Corollary}

\theoremstyle{definition}
\newtheorem{Definition}[Theorem]{Definition}

\newtheorem{Remark}[Theorem]{Remark}

\newcommand{\bTheorem}[1]{
\begin{Theorem} \label{T#1} }
\newcommand{\eT}{\end{Theorem}}

\newcommand{\bProposition}[1]{
\begin{Proposition} \label{P#1}}
\newcommand{\eP}{\end{Proposition}}

\newcommand{\bLemma}[1]{
\begin{Lemma} \label{L#1} }
\newcommand{\eL}{\end{Lemma}}
\newcommand{\bCorollary}[1]{
\begin{Corollary} \label{C#1} }
\newcommand{\eC}{\end{Corollary}}

\renewcommand{\(}{\left(}
\renewcommand{\)}{\right)}

\allowdisplaybreaks

\renewcommand{\u}{\mathbf{u}}
\newcommand{\vv}{\mathbf{v}}

\newcommand{\J}{\mathbf{J}}
\newcommand{\Jh}{\skew{2}\hat{\mathbf{J}}}
\newcommand{\bxi}{\boldsymbol{\xi}}
\newcommand{\LL}{L^2(0,T;L^2(\Omega))}

\newcommand{\CC}{\mathbf{C}}
\newcommand{\DD}{\mathbf{D}}

\newcommand{\I}{\mathbf{I}}
\newcommand{\TT}{\mathbf{T}}
\newcommand{\tr}[1]{\mathrm{tr}\({#1}\)}

\newcommand{\trC}{\mathrm{tr}(\CC)}
\newcommand{\Du}{\mathrm{D}\u}
\renewcommand{\div}[1]{\mathrm{div}\left({#1}\right)}
\newcommand{\di}[1]{\mathrm{div}\,{#1}}
\newcommand{\dib}[1]{\mathrm{div}\big({#1}\big)}
\newcommand{\p}{\partial}
\newcommand{\na}{\nabla}
\newcommand{\R}{\mathbb{R}}

\newcommand{\td}{\frac{\mathrm{d}}{\mathrm{dt}}}
\newcommand{\dx}{\,\mathrm{d}x}

\newcommand{\dt}{\,\mathrm{d}t}
\newcommand{\dta}{\,\mathrm{d}t^\prime}

\DeclarePairedDelimiter{\norm}{\|}{\|}
\DeclarePairedDelimiter{\snorm}{|}{|}

\newcommand\restr[2]{\ensuremath{\left.#1\right|_{#2}}}

\def\softd{{\leavevmode\setbox1=\hbox{d}%
          \hbox to 1.05\wd1{d\kern-0.4ex{\char039}\hss}}}
\def\softt{{\leavevmode\setbox1=\hbox{t}%
			\hbox to 1.05\wd1{t\kern-0.4ex{\char039}\hss}}}

\date{}

\begin{document}

\title{Global existence of weak solutions to viscoelastic\\ phase separation: Part II Degenerate  Case}

\author{Aaron Brunk \and M\' aria Luk\' a\v cov\' a-Medvi{\softd}ov\'a
}

\date{\today}

\maketitle

\bigskip

\centerline{ Institute of Mathematics, Johannes Gutenberg-University Mainz}

\centerline{Staudingerweg 9, 55128 Mainz, Germany}

\centerline{abrunk@uni-mainz.de}
\centerline{lukacova@uni-mainz.de}

\begin{abstract}
\noindent The aim of this paper is to prove  global in time existence of weak solutions for a viscoelastic phase separation. We consider the case with singular potentials and degenerate mobilities. Our model couples the diffusive interface model with the Peterlin-Navier-Stokes equations for viscoelastic fluids.
To obtain the global in time existence of weak solutions we consider appropriate approximations by solutions of the viscoelastic phase separation with a regular potential and build on the corresponding energy and entropy estimates.
\end{abstract}



\section{Introduction}
 For Newtonian fluids the phase separation of binary fluids is a quite well-understood process.
 On the macroscopic scale the so-called H model \cite{Hohenberg.1977} is typically used, which consists of the conservation of mass and momentum coupled with a nonlinear convection-diffusion equation describing the dynamics of a phase variable $\phi$. The evolution of the phase variable is mainly driven by a gradient flow for the free energy functional. In this context the conservation of the mixture is essential, which means that there is no nucleation or phase transition between the components. In order to avoid discontinuous interfaces a typical choice is the diffusive interface approach which consists of penalizing the mixing energy  by a gradient of the phase variable
 $\phi$. A typical choice of the phase variable $\phi$ is the volume fraction which yields  $\phi\in [0,1]$. \\

Taking viscoelastic effects of the material into account leads to a new challenging problem, since we have to consider additionally multiphase viscoelastic effects. The term \emph{viscoelastic phase separation} has been introduced by Tanaka \cite{Tanaka.}.
{\cred Viscoelastic phase separation governs
phase separation of a dynamically asymmetric mixture, which is composed of fast and slow
components. In dynamically asymmetric mixtures the phase separation leads to new interesting
structures, such as transient formation of network-like structures of a slow-component-rich phase
and its volume shrinking. Tanaka's model couples the H-model with the evolution equation for the viscoelastic stress tensor and the  equation for
the bulk stress.} Unfortunately, the model of Tanaka is not thermodynamically consistent. This drawback has been cured by
Zhou, Zhang, and E who  derived in \cite{Zhou.2006} a thermodynamically consistent model for viscoelastic phase separation. The latter allows the cross-diffusive structure between the phase variable and the bulk stress. In our recent  work \cite{Brunk.} we have analysed the model of Zhou, Zhang, and E for the regular Ginzburg-Landau potential, which is of polynomial type, and proved the existence of global weak solutions.
\\

In the literature one can find other approaches to model dynamics of solvent-polymer mixtures. We refer to
the works of Gr\"un and Metzger~\cite{Grn2016,Grn2017} where a micro-macro model consisting of the
Fokker-Planck type equation describing distribution and orientation of the polymer chains
and the Cahn–Hilliard-Navier–Stokes type equations describing the balance
of mass and momentum was studied and the existence of global weak solutions and some
numerical experiments were presented. In the case of the Hookean spring potential the
macroscopic limit was derived. In \cite{metzger2018} Metzger provides rigourous numerical analysis for a finite element approximation of the above micro-macro model for polymeric flows.  Further generalisations
including different phase densities, varying interface parameters or compressible effects are studied in \cite{Abels.2013,Abels.2008,ABELS2012, Lowengrub.1998}.\\

Due to the conservation of the mixture the convection-diffusion equation describing the evolution of the volume fraction is of the Cahn-Hilliard-type. The latter contains a fourth order differential operator and  does not obey a maximum principle, i.e.~in the case of regular potentials $\phi$ could leave the interval $[0,1]$. Therefore it is crucial to find a maximum-like principle to ensure that the volume fraction remains in the physically reasonable interval $[0,1]$. This leads to the case of
singular potentials, such as the Flory-Huggins potential with the logarithmic behaviour
and the degenerate mobilities. In other words, the viscoelastic phase separation model that is based on the singular potentials and/or degenerate mobilities better reflects important physical properties.\\
{\cred
 In degenerate situations the main mathematical difficulty arises from losing a-priori control of the chemical potential and possible singularities of the potential. We refer to \cite{agosti_cahn-hilliard-type_2017, Boyer.1999, cances_two-phase_2019, dai_weak_2021, Elliott.1996, Grun.1995, jihui_degenerate_2020,
liu_convective_2008,Passo1998}, where similar situations for the Cahn-Hilliard-type models are studied.}\\

The main goal of this paper is to extend and generalize our recent results on the viscoelastic phase separation model with the regular potential \cite{Brunk.}
by considering a \emph{singular potential and degenerate mobilities}. Our aim is to show \emph{existence of global weak solutions}. We should note that the model of Zhou et al.~\cite{Zhou.2006} has been already successfully used for numerical simulations of spinodal decomposition, see, e.g.~\cite{LukacovaMedvidova.2017}, \cite{Strasser.2018} for both regular and singular potentials, however the mathematical analysis was still open. Thus, the aim of the present paper is to fill this gap and present a rigorous analysis of the viscoelastic phase separation model in the case of a singular potential and degenerate mobilities. \\

The paper is organized in the following way. In Section 2 we present the mathematical model for viscoelastic phase separation. The weak solution to our viscoelastic phase separation model is introduced in Section 3 and the main result on the existence of global weak solutions is formulated in Section 4. Sections 5-8 are devoted to the proof of our main result. In Section 9 we study the situation of two different degenerate mobilities. Section 10 presents results of three-dimensional numerical simulations for the spinodal decomposition problem. {\cred Numerical results are in a good agreement with physical experiments of dynamically asymmetric mixtures, see \cite{Tanaka.}. Specifically, formation of  transient network-like polymeric structures and volume shrinking can been observed as expected.}

\section{Mathematical Model}
The viscoelastic phase separation can be described by a coupled system consisting of the Cahn-Hilliard equation for phase field evolution, the Navier-Stokes equation for fluid flow and the time evolution of the viscoelastic conformation tensor. The total energy of the polymer-solvent mixture consists of the mixing energy between the polymer and the solvent, the bulk stress energy, the elastic energy and the kinetic energy.
\begin{align}
E_{tot}(\phi,q,\CC,\u)&=  E_{mix}(\phi) + E_{bulk}(q) + E_{el}(\CC) + E_{kin}(\u) \label{eq:free_energy}\\
&=\int_\Omega \(\frac{c_0}{2}\snorm*{\na\phi}^2 + F(\phi)\) + \int_\Omega \frac{1}{2}q^2 + \int_\Omega \(\frac{1}{4}\tr{\TT - 2\ln(\CC) - \I}\) + \int_\Omega \frac{1}{2}\snorm*{\u}^2, \nonumber
\end{align}
where $\phi$ denotes the volume fraction of polymer molecules, $q\I$ the bulk stress {\cred arising from (microscopic) intermolecular attractive interactions}, $\CC$ the viscoelastic conformation tensor and $\u$ the volume averaged velocity consisting of the solvent and polymer velocity. {\cred For a detailed physical explanation of the model we refer to \cite{Tanaka., Zhou.2006}, see also \cite{brunk2020analysis, Brunk.} for a GENERIC-type consistent derivation.  We note that $q$ can be seen as a viscoelastic pressure which is related to the divergence of the relative velocity, while $\CC$ is related to the deformation of the divergence-free volume-averaged velocity $\u$.} Furthermore, $c_0$ is a positive constant proportional to the interface width. In the present work we will focus on the Flory-Huggins potential,
\begin{equation}
\label{eq:fhpot}
F(\phi) = \frac{1}{n_p}\phi\ln(\phi) + \frac{1}{n_s}(1-\phi)\ln(1-\phi) + \chi\phi(1-\phi),
\end{equation}
where $n_p,n_s$ stand for the molecular weights of the polymer and solvent, respectively. Here, $\chi$ is the so-called Flory interaction parameter which describes the interaction between the mixture components. In what follows we will study a class of more general potentials that includes (\ref{eq:fhpot}).
A full derivation of the model can be found in \cite{Brunk.}.

\begin{tcolorbox}
	\begin{align}
		\label{eq:full_model}
	\begin{split}
	\frac{\partial \phi}{\partial t} + \hspace{1em}\u\cdot\nabla\phi  &= \dib{m(\phi)\nabla\mu} - \dib{n(\phi)\nabla\big(A(\phi)q\big)} \\
	\frac{\partial q   }{\partial t} + \hspace{1em}\u\cdot\nabla q    &= -\frac{1}{\tau(\phi)}q + A(\phi)\Delta\big(A(\phi)q\big) - A(\phi)\dib{n(\phi)\nabla\mu} + \varepsilon_1\Delta q\\
	\frac{\partial\u   }{\partial t} + (\u\cdot\nabla)\u  &= \frac{1}{2}\dib{\eta(\phi)\big(\nabla\u + (\nabla\u)^\top\big)} -\nabla p + \di{\TT}  + \nabla\phi\mu \\
	\frac{\partial\CC  }{\partial t} + (\u\cdot\nabla)\CC &= (\nabla\u)\CC + \CC(\nabla\u)^\top + \trC\,\mathbf{I} - \trC^2\,\CC + \varepsilon_2\Delta\CC \\
	\div{\u} &= 0 \hspace{1cm}\TT = \tr{\CC}\CC - \I \hspace{1cm}   \mu = -c_0\Delta\phi + f(\phi)
	\end{split}
	\end{align}
\end{tcolorbox}
System (\ref{eq:full_model}) is formulated on $(0,T)\times\Omega$, where $\Omega\subset\mathbb{R}^2$ has a convex Lipschitz-continuous boundary or is at least $C^{1,1}$. It is equipped with the following initial and boundary conditions
\begin{align}
\restr{(\phi,q,\u,\CC)}{t=0} = (\phi_0 , q_0, \u_0, \CC_0),  \restr{\p_n\phi}{\p\Omega}=\restr{\p_n\mu}{\p\Omega}=\restr{\p_nq}{\p\Omega}=0,  \restr{\u}{\p\Omega} = \mathbf{0}, \restr{\p_n\CC}{\p\Omega} = \mathbf{0}. \label{eq:sysbc}
\end{align}

We consider the relation between the conformation tensor and the viscoelastic stress tensor given by
\begin{equation*}
\TT = \tr{\CC}\CC - \I,
\end{equation*}
see, e.g. \cite{LukacovaMedvidova.2015,LukacovaMedvidova.2017d}.
We proceed by formulating some assumptions on system (\ref{eq:full_model}).

{\cred
\begin{Remark}
In most of the rheological literature on viscoelastic fluids
non-diffusive equations with $\varepsilon_1 = 0 = \varepsilon_2$ are considered. A
justification for strictly positive coefficients $\varepsilon_1, \varepsilon_2$ has been given in \cite{Barrett2007}, where it is explained
that these diffusive terms arise due to the center-of-mass diffusion of the polymer chains. Furthermore, the equations for $q$ and $\CC$ are not coupled, due to its inherent different association with the relative and the volume-averaged velocity, respectively. If the higher order deformation tensors associated to the relative velocity are included the coupling between these quantities can be revealed.
\end{Remark}
}

\subsection{Assumptions}
\label{sub:ass}
We assume that the coefficient functions $\tau,h,\eta$ are continuous, positive and bounded, i.e.
\begin{equation}
\label{eq:coeff1}
0 < \tau_1 \leq \tau(x) \leq \tau_2, \quad 0 < h_1\leq h(x) \leq h_2, \quad 0 < \eta_1 \leq \eta(x) \leq \eta_2 \quad \textrm{ for all } x\in \mathbb{R}.
\end{equation}
The functions $A,m,n$ are assumed to be continuous, non-negative and bounded,
\begin{equation}
\label{eq:coeff2}
0 \leq A_1 \leq A(x) \leq A_2,\quad 0 \leq m_1\leq m(x) \leq m_2, \quad 0 \leq n_1\leq n(x) \leq n_2 \quad \textrm{ for all } x\in \mathbb{R}.
\end{equation}
Further $A(x)$ is a $C^1(\mathbb{R})$ function with $\norm*{A^\prime}_{L^\infty} \leq A_2^\prime$.
In order to differentiate between the regular case and the case with degenerate mobility functions $m$ and $n$ we introduce the following assumptions.

\begin{Assumptions}[Regular case] \label{ass:regular}
\hspace{1em} \\
\vspace{-1em}
 \begin{itemize}
  \item We assume  $m(x)=n(x)\in C(\mathbb{R})$ with $0 < m_1\leq m(x)\leq 1,$ for all $ x\in\mathbb{R}.$
  \item We assume $F\in C^2(\mathbb{R})$ with constants $c_i > 0,\; i=1,\ldots,7$ and $c_8\geq 0$ such that:
  \begin{align*}
|F(x)| \leq c_1|x|^{p} + c_2&, \hspace{1em}|F^{\prime}(x)| \leq c_3|x|^{p-1} + c_4, \hspace{1em}|F^{\prime\prime}(x)| \leq c_5|x|^{p-2} + c_6\textnormal{ for } p\geq 2, \\
F(x) \geq -c_7&, \quad F^{\prime\prime}(x) \geq - c_8.
\end{align*}
 \end{itemize}
\end{Assumptions}
A typical example satisfying the above assumptions is the well-known Ginzburg-Landau potential $F(x)=x^2(1-x)^2$.
For the single degenerate mobility case we introduce the following set of assumptions.
\begin{Assumptions}[Single Degenerate case] \label{ass:deg}
\hspace{1em} \\
\vspace{-1em}
 \begin{itemize}
  \item We assume $m\in C^1([0,1])$ with $m(x) = 0$ if and only if $x\in\{0,1\}.$ The mobility function $m(x)$ is continuously extended by zero on $\mathbb{R}\setminus[0,1]$.
  \item Further, we set $n(x)=m(x)$ and $m(x)\leq 1$, $x\in [0,1].$
  \item The potential can be divided into $F=F_1+F_2$ with a convex part $F_1\in C^2([0,1])$ and a concave part $F_2\in C^2([0,1])$. $F_2$ is continuously extended on $\mathbb{R}$ such that $\norm*{F_2^{\prime\prime}}_{L^\infty(\R)}\leq F_0$.
  \item The convex part $F_1$ additionally satisfies $mF_1^{\prime\prime}\in C([0,1])$.
  \item The boundary condition $\restr{m(\phi)\na\mu\cdot\mathbf{n}}{\p\Omega}=0$ holds.
 \end{itemize}
\end{Assumptions}
The typical example satisfying the above assumptions $m(x)=x^m(1-x)^m$ and the Flory-Huggins potential (\ref{eq:fhpot}) for suitable $\chi$ and $m\geq 1$.
The assumptions of the case with different mobilities differ only in the choice of mobility and the function $A$.
\begin{Assumptions}[Double degenerate case] \label{ass:deg2}
	\hspace{1em} \\
	\vspace{-1em}
	\begin{itemize}
\item We assume $n,m\in C^1([0,1])$ with $n(x) = m(x) = 0$ if and only if $x\in\{0,1\}.$ The mobility functions $n(x), m(x)$ are continuously extended by zero on $\mathbb{R}\setminus[0,1]$.
\item We set $n^2(x)=m(x)$ with $m(x)\leq m_2$ for $x \in [0,1]$.
 \item The potential can be divided into $F=F_1+F_2$ with a convex part $F_1\in C^2(0,1)$ and a concave part $F_2\in C^2([0,1])$. $F_2$ is continuously extended on $\mathbb{R}$ such that $\norm*{F_2^{\prime\prime}}_{L^\infty(\R)}\leq F_0$.
\item The convex part $F_1$ additionally satisfies $(mF_1^{\prime\prime})\in C([0,1])$.
\item We assume $\norm*{\frac{A(x)}{n(x)}}_{L^\infty(0,1)}\leq c,\quad \norm*{\frac{A^\prime(x)}{n(x)}}_{L^\infty(0,1)}\leq c.$
  \item The boundary condition $\restr{m(\phi)\na\mu\cdot\mathbf{n}}{\p\Omega}=0$ holds.
	\end{itemize}
\end{Assumptions}

\begin{Lemma}[\cite{Grun.1995}]
	\label{lema:gtime}
	Let $\phi\in L^2(0,T;H^1(\Omega))$, $\frac{\partial\phi}{\partial t}\in L^2(0,T;H^{-1}(\Omega))$ and let $G\in C^2(\mathbb{R},\mathbb{R})$ be convex with linear growth of its derivative $G^\prime$. Then
	\begin{equation*}
	\int_{r_1}^{r_2} \left\langle \frac{\partial \phi}{\partial t}, G^\prime(\phi) \right\rangle \dt = \int_\Omega G(\phi(x,r_2)) \dx - \int_\Omega G(\phi(x,r_1)) \dx
	\end{equation*}
	for arbitrary $r_1,r_2\in[0,T)$.
\end{Lemma}

Now, let us introduce the following notation $H:=\{\vv\in L^2(\Omega): \div{\vv}=0\}$ and $V:=H^1_0(\Omega)\cap H.$	
In order to prove our main result on the global existence of a weak solution for model (\ref{eq:full_model}) with the degenerate mobilities and singular potential function we first recall our existence result for the regular case.
The following theorem can be proven analogously as in \cite{Brunk.}.
\begin{Theorem}
	\label{theo:ex}
	Let Assumption \ref{ass:regular} are fulfilled and the initial data $$\left(\phi_0,q_0,\u_0,\CC_0 \right) \in H^1(\Omega)\times L^2(\Omega) \times H \times L^2(\Omega)^{2\times2}.$$ Then for a given $T>0$ there exists a global weak solution $(\phi,q,\mu,\u,\CC)$ of (\ref{eq:full_model}) such that	
	\begin{align*}
	&\phi \in L^{\infty}(0,T;H^1(\Omega))\cap L^2(0,T;H^2(\Omega)),& & q,\CC\in L^\infty(0,T;L^2(\Omega))\cap L^2(0,T;H^1(\Omega)),\\
	&\u\in L^{\infty}(0,T;L^2(\Omega))\cap L^2(0,T;V),& & \mu,A(\phi)q \in L^2(0,T;H^1(\Omega)),
	\end{align*}
    and
	\begin{align*}
	&\phi^\prime \equiv \frac{\partial \phi}{\partial t} \in L^{2}(0,T;(H^{1}(\Omega))^*),& &
q^\prime \equiv \frac{\partial q}{\partial t},\CC^\prime \equiv \frac{\partial \CC}{\partial t}\in L^{4/3}(0,T;(H^{1}(\Omega))^*),&
&\u^\prime \equiv \frac{\partial u}{\partial t} \in L^2(0,T;V^{*}).
	\end{align*}
Further, for any test function $(\psi,\zeta,\xi,\vv,\DD)\in H^1(\Omega)^3\times V \times H^1(\Omega)^{2\times2}$ and almost every $t\in(0,T)$
it holds
	\begin{align}
	&\int_\Omega\frac{\partial \phi}{\partial t}\psi \dx + \int_\Omega\(\u\cdot\na\phi\)\psi\dx + \int_\Omega m(\phi)\nabla\mu\cdot\nabla\psi  \dx - \int_\Omega m(\phi)\nabla\big(A(\phi)q\big)\cdot\na\psi \dx = 0 \nonumber\\
	&\int_\Omega\frac{\partial q   }{\partial t}\zeta \dx + \int_\Omega\(\u\cdot\na q\)\zeta\dx + \int_\Omega\frac{q\zeta}{\tau(\phi)}\dx + \int_\Omega\nabla\big(A(\phi)q\big)\cdot\na\big(A(\phi)\zeta\big)\dx + \int_\Omega \varepsilon_1\nabla q \cdot\nabla\zeta \dx\nonumber \\
	& \hspace{7em} = \int_\Omega m(\phi)\nabla\mu\cdot\na\big(A(\phi)\zeta\big) \dx \nonumber\\
	&\int_\Omega\mu\xi \dx - c_0\int_\Omega\nabla\phi\cdot\nabla\xi \dx - \int_\Omega F^{\prime}(\phi)\xi \dx = 0 \label{eq:weak_reg}\\
	&\int_\Omega\frac{\partial\u   }{\partial t}\cdot\vv \dx + \int_\Omega(\u\cdot\na)\u\cdot\vv \dx - \int_\Omega\eta(\phi)\nabla\u:\nabla\vv \dx + \int_\Omega\TT:\nabla\vv \dx - \int_\Omega\mu\nabla\phi\cdot\vv \dx = 0 \nonumber\\
	&\int_\Omega\frac{\partial\CC  }{\partial t}:\DD \dx + \int_\Omega(\u\cdot\na)\CC:\DD \dx- \int_\Omega \Big[(\nabla\u)\CC + \CC(\na\u)^T\Big]:\DD \dx + \varepsilon_2\int_\Omega\nabla\CC:\nabla\DD \dx  \nonumber\\
	& \hspace{7em} =\int_\Omega h(\phi)\trC\I:\DD \dx - \int_\Omega h(\phi)\tr{\CC}^2\CC:\DD \dx  \nonumber
	\end{align}
and the initial conditions $\left(\phi(0),q(0),\u(0),\CC(0) \right) = \left(\phi_0,q_0,\u_0,\CC_0 \right)$ are fulfilled.\\
	Moreover, the weak solution satisfies the energy inequality
	\begin{align}
	&\left(\int_\Omega \frac{c_0}{2}\snorm{\nabla\phi(t)}^2 + F(\phi(t)) + \frac{1}{2}|q(t)|^2 +\frac{1}{2}\snorm{\u(t)}^2 + \frac{1}{4}\snorm*{\CC(t)}^2 \dx\right) \label{eq:regularenergy} \\
	&\leq -\int_0^t\int_\Omega \(\snorm*{\sqrt{m(\phi)}\nabla\mu} - \snorm*{\na\big(A(\phi)q\big)}\)^2\dx\dta -\int_0^t\int_\Omega \frac{1}{\tau(\phi)}q^2 \dx\dta -\varepsilon_1\int_0^t\int_\Omega \snorm*{\nabla q}^2 \dx\dta\nonumber \\
	&-\int_0^t\int_\Omega \eta(\phi)\snorm*{\Du}^2\dx\dta - \frac{\varepsilon_2}{2}\int_0^t\int_\Omega\snorm*{\nabla\CC}^2\dx\dta - \frac{1}{2}\int_0^t\int_\Omega h(\phi)\snorm*{\trC\CC}^2\dx\dta \nonumber  \\
	&+ \frac{1}{2}\int_0^t\int_\Omega h(\phi)\snorm*{\trC}^2\dx\dta+\left(\int_\Omega \frac{c_0}{2}\snorm{\nabla\phi_0}^2 + F(\phi_0) + \frac{1}{2}|q_0|^2 +\frac{1}{2}\snorm{\u_0}^2 + \frac{1}{4}\snorm*{\CC_0}^2 \dx\right) \nonumber
	\end{align}
	for almost every $t\in(0,T)$.
\end{Theorem}

\begin{proof}[Sketch of the proof]
	We will shortly outline the main ideas to prove the above result. Applying the energy method we obtain (\ref{eq:regularenergy}) and consequently
	\begin{align}
	\u \in L^\infty(0,T;L^2(\Omega))\cap L^2(0,T;V) \quad q,\CC \in L^\infty(0,T;L^2(\Omega))\cap L^2(0,T;H^1(\Omega)) \nonumber \\
	\phi \in L^\infty(0,T;H^1(\Omega)), \qquad \snorm*{\sqrt{m(\phi)}\na\mu} - \snorm*{\na\big(A(\phi)q\big)} \in L^2(0,T;L^2(\Omega)). \label{eq:gron1}
	\end{align}
	The key step is now to obtain $m(\phi)$ independent estimates. To this end we construct the so called entropy function \cite{Boyer.1999,Elliott.1996} $G$ via
	\begin{equation}
	G(1/2) = 0, \quad G^\prime(1/2) = 0 \quad G^{\prime\prime}(y) = \frac{1}{m(y)}, \text{ for all } y\in\mathbb{R}. \label{eq:defentropy}
	\end{equation}
	Since $m(y)>0$ for all $y\in\mathbb{R}$, the function $G(y)$ is convex and non-negative by construction.
	Using $G^\prime(\phi)$ as a test function in $(\ref{eq:weak_reg})_1$ and applying Lemma \ref{lema:gtime} yields
	\begin{align*}
	\td \int_\Omega G(\phi) \dx + \int_\Omega \(\u\cdot\na\phi\) G^\prime(\phi) \dx = -\int_\Omega m(\phi)\na\mu\cdot \na G^\prime(\phi) \dx + \int_\Omega n(\phi)\na\big(A(\phi)q\big)\cdot\na G^\prime(\phi) \dx.
	\end{align*}
	Recalling that $\div{\u}=0$ and $G^{\prime\prime}(\phi)m(\phi)=1$ by (\ref{eq:defentropy}) we deduce
	\begin{align*}
	\td \int_\Omega G(\phi) \dx = -\int_\Omega \na\mu\cdot \na\phi \dx + \int_\Omega \frac{n(\phi)}{m(\phi)}\cdot\na\big(A(\phi)q\big)\na\phi \dx.
	\end{align*}
	Being in the regular case, i.e. $n(\phi)/m(\phi)=1$ we can estimate the last integral as follows
	\begin{align}
	\label{eq:Acontrol}
	\int_\Omega \na\big(A(\phi)q\big)\na(\phi) \dx \leq c(c_0,A,A^\prime,\sigma)(\norm*{q}_{L^2}^2\norm*{\na\phi}_{L^2}^2 + \norm*{\na\phi}_{L^2}^2\norm*{\na q}_{L^2}^2) +  c_0\sigma\norm*{\Delta\phi}_{L^2}^2,
	\end{align}
	where $\sigma < 1$ and $c\approx \sigma^{-1}$ are constants from the Young inequality.
	Here we have used the notation $\norm*{\cdot}_{L^2}$ for the norm in $L^2(\Omega)$.
	Integration by parts and inserting the definition of $\mu$, cf. $(\ref{eq:weak_reg})_3$, yields
	\begin{align}
	\td \int_\Omega G(\phi) \dx + (1-\sigma)c_0\int_\Omega |\Delta\phi|^2 \dx + \int_\Omega F^{\prime\prime}(\phi)|\na\phi|^2\dx \leq c\norm*{\na\phi}_{L^2}^2(\norm*{q}_{L^2}^2 + \norm*{\na q}_{L^2}^2). \label{eq:regularlaplaceregularity}
	\end{align}
	Due to (\ref{eq:gron1}) we know that the right hand side of (\ref{eq:regularlaplaceregularity}) is bounded in $L^1(0,T)$ and we can apply the Gronwall inequality to find the a priori bounds
	\begin{align}
	\norm*{\int _\Omega G(\phi) \dx}_{L^\infty(0,T)} \leq c, \quad\quad \norm*{\Delta\phi}_{L^2(L^2)}\leq c, \label{eq:gron3}
	\end{align}
	where $\norm*{\cdot}_{L^p(L^q)}$ is the norm in the space $L^p(0,T;L^q(\Omega)), 1\leq p,q \leq \infty.$
	Next we calculate the $L^2$-norm of $\na\big(A(\phi)q\big)$ as follows
	\begin{align}
	\int_0^T \int_\Omega \snorm*{\na\big(A(\phi)q\big)}^2 \dx \dt &\leq c\int_0^T \int_\Omega \snorm*{A^\prime(\phi)\na\phi q}^2 + \snorm*{A\na q}^2 \dx\dt \nonumber\\
	&\leq cA_2^2\norm*{\na q}_{L^2(L^2)}^2 + c\norm*{A^\prime}^2_{L^\infty}\int_0^T \norm*{\na\phi}_{L^4}^2\norm*{q}^2_{L^4}\dt \nonumber\\
	&\leq cA_2^2\norm*{\na q}_{L^2(L^2)}^2 + c\norm*{A^\prime}^2_{L^\infty}\int_0^T \norm*{\na\phi}_{L^2}^2\norm*{q}^4_{L^4} + \norm*{\Delta\phi}_{L^2}^2 \dt \nonumber\\
	&\leq c\(\norm*{\na q}_{L^2(L^2)}^2 + \norm*{\Delta\phi}_{L^2(L^2)}^2 + \norm*{\na\phi}^2_{L^\infty(L^2)}\norm*{q}_{L^4(L^4)}^4 \)\leq c. \label{eq:Al2}
	\end{align}
	
	Combining estimate (\ref{eq:gron1}) with (\ref{eq:gron3}) we recover the bound on $\sqrt{m(\phi)}\nabla\mu$ in $L^2(0,T;L^2(\Omega))$, which implies $\nabla\mu \in L^2(0,T;L^2(\Omega))$. With this information one can follow \cite{Brunk.} to obtain the existence result.
	
\end{proof}

\section{Weak solution}
In this section we will introduce the notion of a weak solution in the case of degenerate mobilities and formulate the corresponding energy inequality.

\begin{Definition}
	\label{defn:weak_sol_deg_full}	
	Let the initial data $$(\phi_0,q_0,\u_0,\CC_0)\in H^1(\Omega)\times L^2(\Omega)\times H\times L^2(\Omega)^{2\times2}.$$ Then for every $T>0$ the quadruple $(\phi,q,\J,\u,\CC)$ is called a weak solution of (\ref{eq:full_model}) if it satisfies
	\begin{align*}	
	&\phi \in L^{\infty}(0,T;H^1(\Omega))\cap L^2(0,T;H^2(\Omega)),& &q,\CC\in L^\infty(0,T;L^2(\Omega))\cap L^2(0,T;H^1(\Omega)), \\
	&\u\in L^{\infty}(0,T;H)\cap L^2(0,T;V),& &\J\in L^2(0,T;L^2(\Omega))
	\end{align*}
	with time derivatives
	\begin{align*}
	&\phi^\prime \equiv \frac{\partial \phi}{\partial t} \in L^{2}(0,T;(H^{1}(\Omega))^*),& &
q^\prime\equiv  \frac{\partial q}{\partial t},\CC^\prime\equiv \frac{\partial \CC}{\partial t}\in L^{4/3}(0,T;(H^{1}(\Omega))^*),&
&\u^\prime \equiv \frac{\partial u}{\partial t}\in L^2(0,T;V^{*}).
	\end{align*}
Further, for any test function $(\psi,\zeta,\bxi,\vv,\DD)\in H^1(\Omega)^2\times H^1(\Omega)\cap L^\infty(\Omega)\times V\times H^1(\Omega)^{2\times2}$ and almost every $t\in(0,T)$ it holds
	\begin{align}
	&\int_\Omega\frac{\partial \phi}{\partial t}\psi \dx + \int_\Omega\(\u\cdot\nabla\phi\)\psi \dx + \int_\Omega \J\cdot\nabla\psi \dx - \int_\Omega m(\phi)\nabla\big(A(\phi)q\big)\cdot\na\psi \dx =0 \nonumber\\
	&\int_\Omega\frac{\partial q}{\partial t}\zeta \dx + \int_\Omega\(\u\cdot\nabla q\)\zeta \dx + \int_\Omega\frac{q\zeta}{\tau(\phi)}\dx +  \int_\Omega\nabla\big(A(\phi)q\big)\cdot\na\big(A(\phi)\zeta\big)\dx + \int_\Omega \varepsilon_1\nabla q\cdot\nabla\zeta\dx    \nonumber\\
	&\hspace{7em} = \int_\Omega \J\cdot\nabla\big(A(\phi)\zeta\big) \dx \nonumber\\
	&\int_\Omega \J\cdot\bxi \dx = c_0\int_\Omega \Delta\phi\div{m(\phi)\bxi}\dx + \int_\Omega m(\phi)F^{\prime\prime}(\phi)\nabla\phi\cdot\bxi \dx\label{eq:weak_sol_sys} \\
	&\int_\Omega\frac{\partial\u}{\partial t}\cdot\vv \dx + \int (\u\cdot\na)\u\cdot\vv \dx + \int_\Omega\eta(\phi)\Du:\mathrm{D}\vv \dx + \int_\Omega\TT:\nabla\vv \dx + \int_\Omega c_0\Delta\phi\nabla\phi\cdot\vv \dx = 0 \nonumber\\
	&\int_\Omega\frac{\partial\CC}{\partial t}:\DD \dx + \int (\u\cdot\na)\CC:\DD \dx - \int_\Omega\Big[(\nabla\u)\CC + \CC(\na\u)^T\Big]:\DD \dx + \varepsilon_2\int_\Omega\nabla\CC:\nabla\DD \dx & \nonumber\\
	&\hspace{7em}=  - \int_\Omega h(\phi)\tr{\CC}^2\CC:\DD \dx + \int_\Omega h(\phi)\tr{\CC}\I:\DD \dx \nonumber&
	\end{align}
and the initial data, i.e. $(\phi(0),q(0),\u(0),\CC(0))=(\phi_0,q_0,\u_0,\CC_0)$ are attained.
\end{Definition}
The above formulation for $\J$ is a weak version of $\J = m(\phi)\big(-c_0\nabla\Delta\phi +F^{\prime\prime}(\phi)\nabla\phi\big)$ thus for a smooth solution we can identify $\J=m(\phi)\nabla\mu$. \\

The system (\ref{eq:full_model}) has a Lyapunov functional which is independent of the positive/negative definiteness of the conformation tensor $\CC$, see, e.g.~\cite{Brunk., Mizerova.2015} for the free energy.
Since this functional will accompany the weak solution we will eliminate the appearance of $\mu$ from the equations. To this end we introduce $\Jh$ such that $\J=\sqrt{m(\phi)}\Jh$ which can be seen as the Cahn-Hilliard flux.

\begin{Lemma}
	Each classical solution of (\ref{eq:full_model}) satisfies the following energy inequality
	\begin{align}
	&\td\left(\int_\Omega \frac{c_0}{2}\snorm{\nabla\phi}^2 + F(\phi) + \frac{1}{2}|q|^2 +\frac{1}{2}\snorm{\u}^2 + \frac{1}{4}\snorm*{\CC}^2 \dx\right) \nonumber \\
	&= - \int_\Omega \(\snorm*{\Jh} - \snorm*{\nabla\big(A(\phi)q\big)}\)^2\dx -\int_\Omega \frac{1}{\tau(\phi)}q^2\dx - \varepsilon_1\int_\Omega \snorm*{\nabla q}^2\dx- \int_\Omega \eta(\phi)\snorm*{\Du}^2\dx \nonumber \\
	& \hspace{0.5cm} - \frac{\varepsilon_2}{2}\int_\Omega\snorm*{\nabla\CC}^2\dx - \frac{1}{2} \int_\Omega h(\phi)\snorm*{\tr{\CC}\CC}^2\dx + \frac{1}{2}\int_\Omega h(\phi)\tr{\CC}^2 \dx  \label{eq:energy_full}
	\end{align}
	The integrated version reads
	\begin{align}
	&\left(\int_\Omega \frac{c_0}{2}\snorm{\nabla\phi(t)}^2 + F(\phi(t)) + \frac{1}{2}|q(t)|^2 +\frac{1}{2}\snorm{\u(t)}^2 + \frac{1}{4}\snorm*{\CC(t)}^2 \dx\right) \nonumber \\
	&\leq -\int_0^t\int_\Omega \(\snorm*{\Jh} - \snorm*{\na\big(A(\phi)q\big)}\)^2\dx\dta -\int_0^t\int_\Omega \frac{1}{\tau(\phi)}q^2 \dx\dta -\varepsilon_1\int_0^t\int_\Omega \snorm*{\nabla q}^2 \dx\dta \nonumber \\
	&-\int_0^t\int_\Omega \eta(\phi)\snorm*{\Du}^2\dx\dta - \frac{\varepsilon_2}{2}\int_0^t\int_\Omega\snorm*{\nabla\CC}^2\dx\dta - \frac{1}{2}\int_0^t\int_\Omega h(\phi)\snorm*{\trC\CC}^2\dx\dta \label{eq:energy_full_int}  \\
	&+ \frac{1}{2}\int_0^t\int_\Omega h(\phi)\snorm*{\trC}^2\dx\dta + \left(\int_\Omega \frac{c_0}{2}\snorm{\nabla\phi_0}^2 + F(\phi_0) + \frac{1}{2}|q_0|^2 +\frac{1}{2}\snorm{\u_0}^2 + \frac{1}{4}\snorm*{\CC_0}^2 \dx\right). \nonumber
	\end{align}
\end{Lemma}



\section{Main Results}
In this section we will formulate the main result on the existence of a global weak solution to the viscoelastic phase separation system (\ref{eq:full_model}) in the degenerate case.
\begin{tcolorbox}
\begin{Theorem}\label{theo:deg1}
	Let Assumptions \ref{ass:deg}  or \ref{ass:deg2} hold. Further let $\phi_0:\Omega\to [0,1]$ and $\phi_0\in H^1(\Omega)$. The potential function $F$ and the entropy function $G$, cf. (\ref{eq:defentropy}), fulfill
	\begin{equation}
	\int_\Omega \Big( F(\phi_0) + G(\phi_0) \dx \Big) < +\infty. \label{eq:startres}
	\end{equation}
	Then for any given $T < \infty$ there exists a global weak solution $(\phi,q,\J,\u,\CC)$ of the viscoelastic phase separation model (\ref{eq:full_model}) in the sense of Definition \ref{defn:weak_sol_deg_full}. Moreover,
	\begin{itemize}
	\item the integrated energy inequality (\ref{eq:energy_full_int}) or (\ref{eq:energy_full_intint}) holds, respectively.
	\item $\phi(x,t) \in [0,1] \text{  for a.e. } (x,t)\in \Omega\times(0,T). $
	\end{itemize}
	If the mobility function satisfies $m^\prime(0) = m^\prime(1) = 0$, then for a.e. $t\in (0,T)$ the set
	\begin{equation*}
	\{(x,t)\in\Omega\times (0,T) \mid \phi(x,t) = 0 \text{ or } \phi(x,t) = 1\}
	\end{equation*}
	has zero measure.
\end{Theorem}
\end{tcolorbox}

\begin{Remark} \hspace{1em}\\
\vspace{-1em}
\begin{itemize}
 \item For the Flory-Huggins potential (\ref{eq:fhpot}) paired with the mobility $m(x)=x(1-x)$ the condition (\ref{eq:startres}) is automatically fulfilled.
 \item The proof also applies for the polynomial potentials $F$ from Assumptions \ref{ass:regular}.
 \item The proof of Theorem \ref{theo:deg1} with Assumptions \ref{ass:deg} will by realized in Sections 5-8 and in Section 9 we give the necessary adjustments for Assumptions \ref{ass:deg2}.
\end{itemize}
\end{Remark}

The rest of the paper is organized in the following way. In Section 5 we introduce a regularized problem that presents an approximation of (\ref{eq:weak_sol_sys}).
In order to pass to the limit with a regularized parameter $\delta \to 0$ some suitable energy and entropy estimates will be derived in Section 6.
The question of the uniform bounds $0 \leq \phi \leq 1$ will be discussed in Section 7. The limiting process is proven in Section 8 and Section 9 is devoted to a discussion of the case with different degenerate mobilities. In order to illustrate properties of the viscoelastic phase separation model we conclude the paper with a series of three-dimensional simulations.


\section{Regularized Problem}
We start by introducing a suitable regularized problem and approximate the degenerate mobility $m$ and the logarithmic potential $F$ by a non-degenerate mobility $m_\delta$ and a smooth potential $F_\delta$ with a parameter $\delta\in (0,\frac{1}{2})$
\begin{equation}
\label{eq:appmob}
m_\delta(x) = \left\{\begin{array}{lr}
m(\delta), &\text{if } x\leq \delta\\
m(x), &\text{if } \delta \leq x \leq 1-\delta \\
m(1-\delta), &\text{if } x \geq 1- \delta.
\end{array}\right.
\end{equation}
Since $F_2$ is already defined on $\mathbb{R}$ and bounded, cf. Assumptions \ref{ass:deg}, we set $F_{2,\delta}=F_2$. Further
\begin{align}
\label{eq:apppot}
F_{1,\delta}(1/2) = F_1(1/2)  \text{ and } F_{1,\delta}^\prime (1/2) = F^\prime_1(1/2), \\
\label{eq:apppot2}
F_{1,\delta}^{\prime\prime}(x) = \left\{\begin{array}{lr}
F_1^{\prime\prime}(\delta), &\text{if } x\leq \delta\\
F_1^{\prime\prime}(x), &\text{if } \delta \leq x \leq 1-\delta \\
F_1^{\prime\prime}(1-\delta), &\text{if } x \geq 1- \delta.
\end{array}\right.
\end{align}
We note that $F_1(x) = F_{1,\delta}(x)$ for $x\in[\delta,1-\delta]$. Next, we define a regularized entropy function $G_\delta$ in the following way
\begin{equation}
\label{eq:appent}
G_\delta(1/2) = 0, \hspace{1em} G^\prime_\delta(1/2) = 0, \hspace{1em} G^{\prime\prime}_\delta(y) = m_\delta(y)^{-1}, \hspace{1em} \text{ for all } y\in\R.
\end{equation}
The system (\ref{eq:full_model}) with $F_\delta$ and $m_\delta$ fulfills the assumptions of Theorem \ref{theo:ex} for every $\delta\in(0,\frac{1}{2})$. Consequently, we have a sequence of regularized weak solutions, denoted by $(\phi_{\delta},q_{\delta},\mu_{\delta},\u_{\delta},\CC_{\delta})$, such that
\begin{align*}
&\phi_{\delta} \in L^{\infty}(0,T;H^1(\Omega))\cap L^2(0,T;H^2(\Omega)),& & q_{\delta},\CC_\delta\in L^\infty(0,T;L^2(\Omega))\cap L^2(0,T;H^1(\Omega)),\\
&\u_{\delta}\in L^{\infty}(0,T;L^2(\Omega))\cap L^2(0,T;V),& &\mu_{\delta},A(\phi_\delta)q_\delta\in L^2(0,T;H^1(\Omega)).
\end{align*}
The family $(\phi_\delta,q_\delta,\u_\delta,\CC_\delta)$ fulfills the weak formulation (\ref{eq:weak_reg}) for every $\delta > 0$, which in the following will be denoted by $(\ref{eq:weak_reg})^\delta$.


\section{Energy and Entropy Estimates}
In order to extract some information about the necessary convergence properties of the approximate solutions we need to derive estimates independent of $\delta$.
Applying the energy inequality (\ref{eq:regularenergy}) we obtain for every $\delta>0$

\begin{align}
\label{eq:energy_app_1}
&\left(\int_\Omega \frac{c_0}{2}\snorm{\nabla\phi_\delta(t)}^2 + F_\delta(\phi_\delta(t))+\frac{1}{2}\snorm{q_\delta(t)}^2 +\frac{1}{2}\snorm{\u_\delta(t)}^2 + \frac{1}{4}\snorm*{\CC_\delta(t)}^2 \dx\right)  \\
\leq&  -\int_0^t\int_\Omega\(\snorm*{\sqrt{m_\delta(\phi_\delta)}\na\mu_\delta} - \snorm*{\na\big(A(\phi_\delta)q_\delta\big)}\)^2 \dx\dta - \int_0^t\int_\Omega\frac{1}{\tau(\phi_\delta)}q_\delta^2\dx\dta  \nonumber\\
& - \varepsilon_1\int_0^t\int_\Omega|\nabla q_\delta|^2\dx\dta -\int_0^t\int_\Omega\eta(\phi_\delta)\snorm*{\Du_\delta}^2\dx\dta - \frac{\varepsilon_1}{2}\int_0^t\int_\Omega|\na\CC_\delta|^2\dx\dta  \nonumber \\
&- \frac{1}{2}\int_0^t\int_\Omega h(\phi_\delta)\snorm*{\tr{\CC_\delta}\CC_\delta}^2\dx\dta + \frac{1}{2}\int_0^t\int_\Omega h(\phi_\delta)\snorm*{\tr{\CC_\delta}}^2\dx\dta \nonumber \\
&+ \left(\int_\Omega \frac{c_0}{2}\snorm{\nabla\phi_\delta(0)}^2 + F_\delta(\phi_\delta(0))+\frac{1}{2}\snorm{q_\delta(0))}^2 +\frac{1}{2}\snorm{\u_\delta(0)}^2 + \frac{1}{4}\snorm*{\CC_\delta(0)}^2 \dx\right). \nonumber
\end{align}
Due to the definition of $F_\delta$, cf. (\ref{eq:apppot}), we can bound the right hand side of (\ref{eq:energy_app_1}) independently of $\delta$. The term $\int_0^t\int_\Omega h(\phi_\delta)\snorm*{\tr{\CC_\delta}}^2$ can be bounded applying the Hölder inequality and taking into account that $\snorm*{\tr{\CC_\delta}}^2\leq d\snorm*{\CC_\delta}^2$. Then the Gronwall lemma implies
\begin{align}
\norm*{\phi_\delta}_{L^\infty(H^1)} +\norm*{q_\delta}_{L^\infty(L^2)} + \norm*{\u_\delta}_{L^\infty(L^2)} +  \norm*{\CC_\delta}_{L^\infty(L^2)} &\leq c, \nonumber\\
\norm*{q_\delta}_{L^2(H^1)} + \norm*{\u_\delta}_{L^2(H^1)} +  \norm*{\CC_\delta}_{L^2(H^1)} &\leq c,\label{eq:est1}\\
\norm*{\snorm*{\sqrt{m_\delta(\phi_\delta)}\na\mu_\delta} - \snorm*{\nabla\big(A(\phi_\delta)q_\delta \big)}}_{L^2(L^2)} &\leq c. \nonumber
\end{align}

Note that $G^{\prime\prime}_\delta$ is bounded and convex and by construction $G_\delta^\prime(\phi_\delta)$ grows like $\phi_\delta$. Applying Lemma~\ref{lema:gtime} we obtain from $(\ref{eq:weak_reg})_1^\delta$ for $\psi=G^\prime_\delta(\phi_\delta)$
\begin{align}
\td\left(\int_\Omega G_\delta(\phi_{\delta})\dx\right) &+ \int_\Omega \(\u_\delta\cdot\nabla\phi_\delta\) G_\delta^\prime(\phi_\delta)\dx  + \int_\Omega m_\delta(\phi_\delta)G^{\prime\prime}_\delta(\phi_\delta)\nabla\mu_\delta\cdot\nabla\phi_\delta \dx \nonumber\\
& = \int_\Omega m_\delta(\phi_\delta)G^{\prime\prime}_\delta(\phi_\delta)\na\big(A(\phi_\delta)q_\delta\big)\cdot\na\phi_\delta \dx. \label{eq:entropy}
\end{align}
Using  $\div{\u_\delta} = 0$ and $m_\delta(\phi_\delta) G_\delta^{\prime\prime}(\phi_\delta) = 1$ we find
\begin{equation*}
\td\left(\int_\Omega G_\delta(\phi_{\delta})\dx\right)  + \int_\Omega \nabla\mu_\delta\cdot\nabla\phi_\delta \dx = \int_\Omega \na\big(A(\phi_\delta)q_\delta\big)\cdot\na\phi_\delta \dx.
\end{equation*}
Including $(\ref{eq:weak_reg})_3^\delta$, integrating by parts and applying the Hölder and Young inequalities yield
\begin{align}
\label{eq:degenchnsp}
&\td\left(\int_\Omega G_\delta(\phi_\delta)\dx\right) + c_0\norm*{\Delta\phi_\delta}_{L^2}^2 + \int_\Omega F_{1,\delta}^{\prime\prime}(\phi_\delta)\snorm*{\nabla\phi_\delta(t)}^2\dx  \nonumber \\
&\leq c(\sigma)\norm*{F_2^{\prime\prime}}_{L^\infty}\norm*{\nabla\phi_\delta}_{L^2}^2 + \int_\Omega \na\big(A(\phi_\delta)q_\delta\big)\cdot\na\phi_\delta \dx,
\end{align}
where $\sigma < 1$ and $c(\sigma)\approx \sigma^{-1}$.
Applying (\ref{eq:Acontrol}) for the last term in (\ref{eq:degenchnsp}) we obtain
\begin{align}
 &\td\left(\int_\Omega G_\delta(\phi_\delta(t))\dx\right) + \frac{c_0}{2}\norm*{\Delta\phi_\delta}_{L^2}^2 + \int_\Omega F_{1,\delta}^{\prime\prime}(\phi_\delta(t))\snorm*{\nabla\phi_\delta(t)}^2\dx \nonumber \\
&\leq c(c_0,A,A^\prime,\sigma)\norm*{\na\phi}_{L^2}^2(\norm*{q}_{L^2}^2 + \norm*{\na q}_{L^2}^2) + c(\sigma)\norm*{F_2^{\prime\prime}}_{L^\infty}\norm*{\nabla\phi_\delta}_{L^2}^2 \label{eq:deltagronwall1}.
\end{align}
We can  see that the right hand side (\ref{eq:deltagronwall1}) is bounded in $L^1(0,T)$ due to a priori estimates (\ref{eq:est1}).
Using the Gronwall lemma and analogous estimates as in (\ref{eq:Al2}) we obtain the following additional estimates independently of $\delta$
\begin{align}
\norm*{\Delta\phi_\delta}_{L^2(L^2)} + \norm*{\nabla\big(A(\phi_\delta)q_\delta\big)}_{L^2(L^2)} + \norm*{\int_\Omega G_\delta(\phi_\delta) \dx}_{L^\infty(0,T)}&\leq c,\label{eq:est2} \\
\text{ which implies }  \norm*{\sqrt{m_\delta(\phi_\delta)}\na\mu_\delta}_2^2 &\leq c.
\end{align}
Due the fact that $m_\delta$ degenerates at $x=0$ and $x=1$ as $\delta\to 0$ we have no information on the weak convergence of $\na\mu_\delta$.
Moreover, following the lines of \cite{Brunk.} we can also find a priori estimates for the time derivatives
\begin{align}
\norm*{\phi_\delta^\prime}_{L^2(H^{-1})}, \hspace{1em}\norm*{q_\delta^\prime}_{L^{4/3}(H^{-1})},\hspace{1em} \norm*{\u_\delta^\prime}_{L^2(V^\prime)}, \hspace{1em}  \norm*{\CC_\delta^\prime}_{L^{4/3}(H^{-1})} &\leq c.
\end{align}

Consequently we have, up to a subsequence, the following convergence results for $\delta\to 0$
\begin{align}
&\phi_\delta \rightharpoonup^{\star} \phi \text{ in } L^\infty(0,T;H^1(\Omega)) & q_\delta &\rightharpoonup^{\star} q \text{ in } L^\infty(0,T;L^2(\Omega)), \nonumber\\
&\phi_\delta \rightharpoonup \phi \text{ in } L^2(0,T;H^2(\Omega)) &q_\delta &\rightharpoonup q \text{ in } L^2(0,T;H^1(\Omega)),\nonumber\\
&\phi^{\prime}_\delta \rightharpoonup \phi^{\prime} \text{ in } L^2(0,T;(H^{1}(\Omega))^*) &q^{\prime}_\delta &\rightharpoonup q^{\prime} \text{ in } L^{4/3}(0,T;(H^{1}(\Omega))^*),\nonumber\\
&\phi_\delta \rightarrow \phi \text{ in } L^2(0,T;H^1(\Omega)) &q_\delta &\rightarrow q \text{ in } L^2(0,T;L^p(\Omega)),	\nonumber\\
&\u_\delta \rightharpoonup^{\star} \u \text{ in } L^\infty(0,T;H) 					&\CC_\delta& \rightharpoonup^{\star} \CC \text{ in } L^\infty(0,T;L^2(\Omega)), 		\nonumber		\\
&\u_\delta \rightharpoonup \u \text{ in } L^2(0,T;V) 			   					&\CC_\delta& \rightharpoonup \CC \text{ in } L^2(0,T;H^1(\Omega)), &\nonumber\\
&\u^{\prime}_\delta \rightharpoonup \u^{\prime} \text{ in } L^2(0,T;V^\star)       	&\CC^{\prime}_\delta& \rightharpoonup \CC^{\prime} \text{ in } L^{4/3}(0,T;(H^{1}(\Omega))^*), &\nonumber\\
&\u_\delta \rightarrow \u \text{ in } L^2(0,T;L^4_{\text{div}}(\Omega)) 						&\CC_\delta& \rightarrow \CC \text{ in } L^2(0,T;L^p(\Omega)). & \label{eq:convergencetable}
\end{align}
Moreover, $\phi_\delta,\nabla\phi_\delta$ converge a.e. in $\Omega\times(0,T)$.
We set $\Jh_\delta = \sqrt{m_\delta(\phi_\delta)}\nabla\mu_\delta$. Due to a priori bound on $\sqrt{m_\delta(\phi_\delta)}\nabla\mu_\delta$ in $L^2(0,T;L^2(\Omega))$ we obtain the weak convergence to some limit $z\in\LL$. Since $\sqrt{m_\delta(\phi_\delta)}$ is absolutely bounded and converge a.e. to $m(\phi)$, see Assumption \ref{ass:deg}, we can easily deduce
\begin{equation}
\label{eq:Jconv}
\Jh_\delta \rightharpoonup \Jh \text{ in } L^2(0,T;L^2(\Omega)), \qquad \J_\delta \rightharpoonup \J  \text{ in } L^2(0,T;L^2(\Omega)).
\end{equation}


\section{$L^\infty$-estimate for the phase variable}
The aim of this section is to derive uniform boundedness of $\phi_\delta$. Indeed, following \cite{Abels.2013,Boyer.1999,Elliott.1996} we can show the following lemma.
{\cred
\begin{Lemma}
Let $\phi_\delta$ be the solution of the regularized problem $(\ref{eq:weak_reg})^\delta_1$. Then
$$
\phi_\delta(x,t) \to  \phi(x,t) \ \mbox{  a.e. } (x,t) \in \Omega\times(0,T) \ \mbox{ uniformly as } \ \delta\to 0.
$$
Furthermore, the limit $\phi(x,t)\in[0,1]$ for a.a $(x,t)\in\Omega\times(0,T).$
\end{Lemma}}
\begin{proof}
By the definition of $G$ we have $G(y)\geq 0, G^\prime(y)\geq 0$ if $y\geq 1/2$ and  $G(y)\geq 0, G^\prime(y)\leq 0$ if $y\leq 1/2$, cf. (\ref{eq:defentropy}). For $y > 1$ it holds
\begin{align*}
G_\delta(y) &= G_\delta(1-\delta) + G^\prime_\delta(1-\delta)\Big(y-(1-\delta)\Big) + \frac{1}{2}G^{\prime\prime}_\delta(1-\delta)\Big(y-(1-\delta)\Big)^2 \\
&= G(1-\delta) + G^\prime(1-\delta)\Big(y-(1-\delta)\Big) + \frac{1}{2}G^{\prime\prime}(1-\delta)\Big(y-(1-\delta)\Big)^2 \\
&\geq \frac{1}{2} \frac{1}{m(1-\delta)}(y-1)^2.
\end{align*}
Analogously, we have for $y < 0$
\begin{align*}
G_\delta(y) &= G_\delta(\delta) + G^\prime_\delta(\delta)\Big(y-\delta\Big) + \frac{1}{2}G^{\prime\prime}_\delta(\delta)\Big(y-\delta\Big)^2 \\
&= G(\delta) + G^\prime(\delta)\Big(y-\delta\Big) + \frac{1}{2}G^{\prime\prime}(\delta)\Big(y-\delta\Big)^2 \\
&\geq \frac{1}{2} \frac{1}{m(\delta)}y^2.
\end{align*}	
These estimates imply two inequalities
\begin{align}
\int\displaylimits_{\Omega\cap \{\phi_\delta > 1\}} (\phi_\delta(x,t) - 1)^2 \dx &\leq 2m(1-\delta)\int_\Omega G_\delta(\phi_\delta(t)) \dx\rightarrow 0, \nonumber\\
\int\displaylimits_{\Omega\cap \{\phi_\delta < 0\}} \phi_\delta(x,t)^2 \dx &\leq 2m(\delta)\int_\Omega G_\delta(\phi_\delta(t)) \dx\rightarrow 0. \label{eq:phi_inf_deg}
\end{align}
The limit in (\ref{eq:phi_inf_deg}) follows from the Vitali Theorem \cite{Folland.2011} realizing that $m_\delta(\delta)\to 0$ and $m_\delta(1-\delta)\to 0$ for $\delta\to 0$. We have shown that
\begin{align*}
\int\displaylimits_{\Omega\cap \{\phi > 1\}} (\phi(x,t) - 1)^2 \dx = 0, \hspace{2em}	\int\displaylimits_{\Omega\cap \{\phi < 0\}} \phi(x,t)^2 \dx  = 0.
\end{align*}
In other words, $\phi(x,t) \leq 1$ and $\phi(x,t) \geq 0$ for a.e. $(x,t)\in\Omega\times(0,T)$.

\end{proof}


\section{Limiting process}
The aim of this section is to pass to the limit $\delta\to 0$ in all relevant terms of the weak formulation $(\ref{eq:weak_reg})^\delta$.
We will only consider the terms which are not covered in \cite{Brunk.}, i.e.~the terms containing $m_\delta, F_\delta$ or $\J_\delta$. First we recall a useful lemma from \cite{Brunk.}.
\begin{Lemma}
	\label{lem:strongconv}
	Suppose that there are two sequences $\{u_n\}_{n=1}^\infty, \{v_n\}_{n=1}^\infty$ and a bounded domain $D$ with the following properties
	\begin{enumerate}
		\item $u_n \rightarrow u$ a.e in $D$ as $n \rightarrow \infty$  and $\norm*{u_n}_\infty \leq c < \infty$ for all $n$
		\item $v_n \rightharpoonup v$ in $L^2(D)$  as $n \rightarrow \infty.$
	\end{enumerate}
	Then the product $u_nv_n$ converges weakly to $uv$ in $L^2(D)$. If $v_n$ is strongly convergent then the results imply strong convergence in $L^2(D)$.

\end{Lemma}
\subsection{Cahn-Hilliard equation}
In the Cahn-Hilliard equation $(\ref{eq:weak_reg})_1^\delta$ we first consider the limiting process in the term
\begin{align*}
T_1(\delta):=\int_0^T\int_\Omega \big(\J_\delta(t) - \J(t)\big)\cdot\na\psi\varphi(t) \dx\dt.
\end{align*}
Due to the weak convergence of $\J_\delta$ in $\LL$, cf. (\ref{eq:Jconv}) we have $T_1(\delta)\to 0$ as $\delta\to 0$. Now we estimate the coupling term with the bulk stress equation $(\ref{eq:weak_sol_sys})_2$.
\begin{align*}
T_2(\delta):=\int_0^T\int_\Omega m_\delta(\phi_\delta(t))\nabla\Big(A(\phi_\delta(t))q_\delta(t)\Big)\cdot\nabla\psi\varphi(t) - m(\phi(t))\nabla\Big(A(\phi(t))q(t)\Big)\cdot\nabla\psi\varphi(t) \dx\dt.
\end{align*}
One can obtain weak convergence of $\nabla\big(A(\phi_\delta)q_\delta \big)$ in $\LL$ by virtue of (\ref{eq:convergencetable}), see \cite{Brunk.} for a proof. Further by Lemma \ref{lem:strongconv} using the continuity and boundedness of $m_\delta(\phi_\delta)$, see Assumption \ref{ass:deg} and (\ref{eq:appmob}) we find that $m_\delta(\phi_\delta)\na\psi\varphi$ converges strongly to $m(\phi) \nabla \psi \varphi$
in $\LL$. Therefore, $T_2(\delta)\to 0$ as $\delta\to 0$.

\subsection{Bulk stress equation}
In the bulk stress equation $(\ref{eq:weak_reg})^\delta_2$ the only new term comes from the coupling to the Cahn-Hilliard equation $(\ref{eq:weak_reg})^\delta_1$
\begin{align*}
T_3(\delta):=\int_0^T\int_\Omega \J_\delta(t)\cdot\nabla\Big(A(\phi_\delta(t))\zeta \Big)\varphi(t) - \J(t)\nabla\cdot\Big(A(\phi(t))\zeta \Big)\varphi(t) \dx\dt .
\end{align*}
Due to the weak convergence of $\J_\delta$ in $\LL$, (\ref{eq:Jconv}) and the strong convergence of $\nabla(A(\phi_\delta)\zeta)$ in $\LL$, which is obtained by virtue of (\ref{eq:convergencetable}) and $\zeta\in H^1(\Omega)$, we deduce that $T_3(\delta)\to 0$ as $\delta\to 0$.

\subsection{Cahn-Hilliard Flux}
First, let us consider the flux terms in the weak formulation, $(\ref{eq:weak_reg})_3^\delta$ and $(\ref{eq:weak_sol_sys})_3$. The left hand side of $(\ref{eq:weak_reg})_3^\delta$ and $(\ref{eq:weak_sol_sys}_3)$ is given by
\begin{align*}
T_4(\delta):=\int_0^T\int_\Omega (\J_\delta(t)-\J)\cdot\bxi\varphi(t).
\end{align*}
Due to the weak convergence of $\J_\delta$ in $\LL$, see (\ref{eq:Jconv}), we deduce $T_4(\delta)\to 0$ as $\delta\to 0$.
Next we continue with the right hand side of $(\ref{eq:weak_reg})_3^\delta$ and $(\ref{eq:weak_sol_sys})_3$
\begin{align*}
T_5(\delta)&:=\int_0^T\int_\Omega \Delta\phi_\delta(t)\mathrm{div}\Big(m_\delta(\phi(t))\bxi\Big)\varphi(t)  - \Delta\phi(t)\mathrm{div}\Big(m(\phi(t))\bxi\Big)\varphi(t) \dx\dt\\
&=T_{5,1}(\delta) + T_{5,2}(\delta).
\end{align*}
Firstly,
\begin{align*}
T_{5,1}(\delta):=\int_0^T\int_\Omega \Delta\phi_\delta(t)m_\delta(\phi(t))\div{\bxi}\varphi(t)  - \Delta\phi(t)m(\phi(t))\div{ \bxi}\varphi(t) \dx\dt
\end{align*}
goes to zero due to the weak convergence of $\Delta\phi_\delta$ in $\LL$, cf. (\ref{eq:convergencetable}), and the strong convergence of $m_\delta(\phi_\delta)\div{\bxi}\varphi$ in $\LL$, cf.~Lemma (\ref{lem:strongconv}) and $\bxi\in H^1(\Omega)\cap L^\infty(\Omega)$.\\[0.5em]
The second term $T_{5,2}(\delta)$ can be treated as follows
\begin{align*}
T_{5,2}(\delta):=&\int_0^T\int_\Omega \Big(m^\prime_\delta(\phi_\delta(t))\Delta\phi_\delta(t)\na\phi_\delta(t) - m^\prime(\phi(t))\Delta\phi(t)\na\phi(t)\Big)\cdot\bxi\varphi(t) \dx\dt \\
= &\int_0^T\int_\Omega \Big(\Delta\phi_\delta(t) - \Delta\phi(t)\Big)\na\phi_\delta(t) m^\prime_\delta(\phi_\delta(t))\cdot\bxi\varphi(t) \dx\dt \\
+ &\int_0^T\int_\Omega \Delta\phi(t)\Big(m^\prime_\delta(\phi_\delta(t))\na\phi_\delta(t) - m^\prime(\phi(t))\na\phi(t)\Big)\cdot\bxi\varphi(t) \dx\dt.
\end{align*}
We realize that due to the weak convergence of $\Delta\phi_\delta$ in $\LL$, cf. (\ref{eq:convergencetable}), and $\bxi\varphi$ in $L^\infty(\Omega\times(0,T))$ we have to study the strong convergence of $m^\prime_\delta(\phi_\delta)\na\phi_\delta$ in $\LL$.  This
already imply the convergence of $T_{5,2}(\delta)\to 0$ as $\delta \to 0.$
Let us consider the strong convergence of $m^\prime_\delta(\phi_\delta)\na\phi_\delta$ and split the integral in the following way
\begin{align}
&\int_0^T\int_\Omega \snorm*{m^\prime_\delta(\phi_\delta(t))\na\phi_\delta(t) - m^\prime(\phi(t))\na\phi(t)}^2 \dx \dt \nonumber\\
= &\int\displaylimits_{\Omega\times(0,T) \cap \{\phi=\{0,1\}\} }\snorm*{m^\prime_\delta(\phi_\delta(t))\na\phi_\delta(t) - m^\prime(\phi(t))\na\phi(t)}^2 \dx\dt \label{eq:mdeltastrong1}\\
+ &\int\displaylimits_{\Omega\times(0,T) \cap \{\phi\in(0,1)\} } \snorm*{m^\prime_\delta(\phi_\delta(t))\na\phi_\delta(t) - m^\prime(\phi(t))\na\phi(t)}^2 \dx\dt  \label{eq:mdeltastrong2}
\end{align}

For $(\ref{eq:mdeltastrong1})$ we note that $\na\phi=0$ a.e. on the set $\{\phi=0 \text{ or } \phi=1\}$, see Lemma 7.7 in \cite{Gilbarg.1977}, and derive
\begin{align*}
\int\displaylimits_{\Omega\times(0,T) \cap \{\phi=\{0,1\}\} } \snorm*{m^\prime_\delta(\phi_\delta(t))\na\phi_\delta(t)}^2 \dx\dt &\leq \norm*{m_\delta^\prime}_\infty^2\int\displaylimits_{\Omega_T \cap \{\phi=\{0,1\}\} } \snorm*{\na\phi_\delta(t)}^2 \dx\dt\\
&\rightarrow c\int\displaylimits_{\Omega\times(0,T) \cap \{\phi=\{0,1\}\} } \snorm*{\na\phi(t)}^2 \dx\dt = 0.
\end{align*}
In the latter  we have used the strong convergence of $\na\phi_\delta$ in $\LL$, cf. (\ref{eq:convergencetable}).	
For the second integral (\ref{eq:mdeltastrong2}) we can conclude the convergence $m^\prime_\delta(\phi_\delta) \to m^\prime(\phi)$ on the set $\{0< \phi < 1\}$. Since the mobility function $m_\delta(\phi_\delta)$ converges a.e. on the set $\{0< \phi < 1\}$ we obtain $m^\prime_\delta(\phi_\delta)\na\phi_\delta$ converges to $m^\prime(\phi)\na\phi$ for a.e $(x,t)\in\Omega\times(0,T)$. This holds due to the strong convergence of $\na\phi_\delta$ in $\LL$, cf. (\ref{eq:convergencetable}). Therefore, (\ref{eq:mdeltastrong2}) goes to zero for $\delta\to 0$ by the generalized Lebesque convergence theorem. Above calculations imply
\begin{equation*}
m^\prime_\delta(\phi_\delta)\na\phi_\delta \rightarrow m^\prime(\phi)\na\phi \text{ in } \LL.
\end{equation*}
Hence, we conclude $T_{5,2}(\delta)\to 0$ as $\delta\to 0$.\\[0.3em]
Next we focus on the integral containing the potential in $(\ref{eq:weak_reg})_3^\delta$ and $(\ref{eq:weak_sol_sys})_3$, which can be rewritten as
\begin{align}
T_6(\delta):=&\int_0^T\int_\Omega \Big( (m_\delta F^{\prime\prime}_\delta)(\phi_\delta(t))\na\phi_\delta(t)  - (m F^{\prime\prime})(\phi(t))\na\phi(t)\Big)\cdot\bxi\varphi(t) \dx\dt \nonumber\\
=& \int_0^T\int_\Omega \Big((m_\delta F^{\prime\prime}_\delta)(\phi_\delta(t)) - (mF^{\prime\prime})(\phi(t))\Big)\na\phi_\delta(t)\cdot\bxi\varphi(t) \dx\dt \label{eq:convpot1} \\
+& \int_0^T\int_\Omega (mF^{\prime\prime})(\phi(t))\Big(\na\phi_\delta(t) - \na\phi(t)\Big)\cdot\bxi\varphi(t) \dx\dt. \label{eq:convpot2}
\end{align}
The second term (\ref{eq:convpot2}) tends to zero due to the bounds on $(mF^{\prime\prime})(\phi)$, cf.~Assumption \ref{ass:deg}, and the strong convergence of $\na\phi_\delta$ in $\LL$, cf.~(\ref{eq:convergencetable}).\\[0.3em]
For the first term (\ref{eq:convpot1}) we have to show that $(m_\delta F^{\prime\prime}_\delta)(\phi_\delta)$ converges to $(mF^{\prime\prime})(\phi)$ a.e. in $\Omega\times(0,T)$. This is clear for $\phi \in (0,1)$ and $\delta$ small enough, because $(mF^{\prime\prime})(\phi_\delta) = (m_\delta F^{\prime\prime}_\delta)(\phi_\delta)$, cf.~Assumption~\ref{ass:deg}. Let us consider the case $\phi = 1$, $\phi = 0$ can be treated analogously.
First, we can see that for $1/2 \leq \phi_\delta(t,x)\leq 1-\delta$ we have the equality
\begin{equation*}
(m_\delta F_\delta^{\prime\prime})(\phi_\delta(t,x)) = (mF^{\prime\prime})(\phi_\delta(t,x)).
\end{equation*}
Second, we observe that for $\phi_\delta > 1-\delta$ we have
\begin{equation*}
(m_\delta F_\delta^{\prime\prime})(\phi_\delta(t,x)) = (mF_{1,\delta}^{\prime\prime})(1-\delta) + m(1-\delta)F_{2}^{\prime\prime}(\phi_\delta(t,x)).
\end{equation*}
Consequently, in both cases we have obtained that
\begin{equation*}
(m_\delta F^{\prime\prime}_\delta(\phi_\delta)) \longrightarrow (mF^{\prime\prime})(\phi) \text{ a.e. in } \Omega\times(0,T).
\end{equation*}
.  \\[0.5em]
\subsection{Navier-Stokes equation}
We continue by discussing the stress term in the Navier-Stokes equation $(\ref{eq:weak_reg})^\delta_4$ and $(\ref{eq:weak_sol_sys})_4$ arising from the coupling to the Cahn-Hilliard equation
\begin{align*}
T_7(\delta):=&\int_0^T\int_\Omega \Big(\Delta\phi_\delta(t)\na\phi_\delta(t) -\Delta\phi(t)\na\phi(t)\Big)\cdot\vv\varphi(t)\dx\dt = T_{7,1}(\delta) + T_{7,2}(\delta) \\
=&\int_0^T\int_\Omega \Big(\Delta\phi_\delta(t) - \Delta\phi(t)\Big)\na\phi_\delta(t)\cdot\vv\varphi(t) \dx\dt + \int_0^T\int_\Omega \Delta\phi(t)\Big(\na\phi_\delta(t)-\na\phi(t)\Big)\cdot\vv\varphi(t) \dx\dt.
\end{align*}
The first term $T_{7,1}(\delta)\to 0$ for $\delta \to 0$, due to the weak convergence of $\Delta\phi_\delta$ in $\LL$, cf.~(\ref{eq:convergencetable}), if  $\na\phi_\delta\vv\varphi\in L^2(0,T;L^2(\Omega))$. Indeed, due to the strong convergence of $\na\phi_\delta$ in $L^2(0,T;L^3(\Omega))$ we obtain  $\na\phi_\delta\vv\varphi\in L^2(0,T;L^2(\Omega))$ taking $\varphi\vv\in L^\infty(0,T;L^6(\Omega))$.
The second term $T_{7,2}(\delta)$ can be controlled in the following way
\begin{align*}
T_{7,2}(\delta):=&\int_0^T\int_\Omega \Delta\phi(t)\Big(\na\phi_\delta(t)-\na\phi(t)\Big)\cdot\vv\varphi(t) \dx\dt \\
&\leq  \max\limits_{t\in [0,T]}\snorm*{\varphi(t)}\norm*{\vv}_6\int_0^T\norm*{\na\phi_\delta - \na\phi}_3\norm*{\Delta\phi}_2 \dt \\
&\leq  \max\limits_{t\in [0,T]}\snorm*{\varphi(t)}\norm*{\vv}_V\norm*{\Delta\phi}_{L^2(L^2)}\norm*{\na\phi_\delta - \na\phi}_{L^2(L^3)}\longrightarrow 0,
\end{align*}
due to the strong convergence of $\na\phi_\delta$ in $L^2(0,T;L^3(\Omega))$. This implies $T_7(\delta)\to 0$ as $\delta\to 0$.\\[0.5em]

\subsection{Stronger degeneracy}
The aim of this section is to show that under some assumptions on the mobility function $m$ we can prove the following result.
\begin{Lemma}
	\label{lem:zeromeas}
	Let $\phi_\delta$ be the solution of the regularized problem $(\ref{eq:weak_reg})^\delta_1$. Further assume that the mobility function satisfies $m^\prime(0)=m^\prime(1)=0$. Then it holds that set \begin{align*}
	\{(x,t)\in\Omega\times(0,T); \phi(x,t)=0 \text{ or } \phi(x,t)=1\}
	\end{align*}
	has zero measure.
\end{Lemma}

\begin{proof}
	
We suppose that the mobility is strongly degenerate, i.e. $m^\prime(1) = m^\prime(0) = 0$. We can see that $G_\delta(y) \rightarrow \infty$ for $y\to 0$ or $y\to 1$ as $\delta\to 0$. Since $\int_\Omega G_\delta(\phi_\delta(t)) \dx$ is bounded in $L^\infty(0,T)$, see (\ref{eq:est2}), we apply the Lemma of Fatou to get
\begin{equation*}
\int_\Omega \lim \inf\limits_{\delta \to 0} G_\delta(\phi_\delta(t)) \dx \leq c \qquad \text{ a.e. } t\in(0,T).
\end{equation*}
We will now discuss three possible cases for $\phi_\delta$ to prove Lemma \ref{lem:zeromeas}.
\begin{enumerate}
	\item [a)] For $\phi_\delta \in (0,1)$ and $\delta$ small enough we have $G_\delta(\phi_\delta(x,t)) = G(\phi_\delta(x,t))$ and by continuity
	\begin{equation*}
	\lim_{\delta \to 0} G_\delta(\phi_\delta(x,t)) = G(\phi(x,t)).
	\end{equation*}
	\item[b)] For $\phi(x,t)=0$ and any $\delta > 0$ we have
		\begin{equation*}
	\min\{G(\delta),G(\phi_\delta(x,t))\} \leq G_\delta(\phi_\delta(x,t)) \to \infty
	\end{equation*}
	because $G(y)\to \infty$ for $y \to 0$.
	\item [c)] For $\phi(x,t)=1$ and any $\delta > 0$ we have
	\begin{equation*}
	\min\{G(1-\delta),G(\phi_\delta(x,t))\} \leq G_\delta(\phi_\delta(x,t)) \to \infty
	\end{equation*}
	because $G(y)\to \infty$ for $y \to 1$.

\end{enumerate}
Combining the above points we have shown that the set $\{x\in\Omega \mid \phi(x,t)=1 \text{ or } \phi(x,t) = 0 \}$ has zero measure for a.a $t\in(0,T)$.
\end{proof}

{\cred
\begin{Remark}
Note that following the approach presented in \cite{Grun.1995,Passo1998} the result of Lemma~\ref{lem:zeromeas} might be refined. Indeed, for stronger degeneracy, i.e. $m\approx s^n$ near zero and $m\approx(1-s)^n$ near one, for $n$ large enough, one can expect that the volume fraction $\phi$ will be confined in $[\alpha,\beta] \subset (0,1)$. This would require additional technical estimates and further assumptions on the potential $f$ and the parametric function $A$. This may be an interesting topic for future research.
\end{Remark}}

\subsection{Limit in the energy inequality}

In order to pass to the limit in the energy inequality $(\ref{eq:energy_app_1})$ we recall that for $\delta > 0$ we have
\begin{align}
E_\delta:=&\left(\int_\Omega \frac{c_0}{2}\snorm{\nabla\phi_\delta(t)}^2 + F_\delta(\phi_\delta(t))+\frac{1}{2}\snorm{q_\delta(t)}^2 +\frac{1}{2}\snorm{\u_\delta(t)}^2 + \frac{1}{4}\snorm*{\CC_\delta(t)}^2 \dx\right)  \\
\leq&  -\int_0^t\int_\Omega\(\snorm*{\sqrt{m_\delta(\phi_\delta)}\na\mu_\delta} - \snorm*{\na\big(A(\phi_\delta)q_\delta\big)}\)^2 \dx\dta - \int_0^t\int_\Omega\frac{1}{\tau(\phi_\delta)}q_\delta^2\dx\dta  \nonumber\\
& - \varepsilon_1\int_0^t\int_\Omega|\nabla q_\delta|^2\dx\dta -\int_0^t\int_\Omega\eta(\phi_\delta)\snorm{\Du_\delta}^2\dx\dta - \frac{\varepsilon_2}{2}\int_0^t\int_\Omega|\na\CC_\delta|^2\dx\dta  \nonumber \\
&- \frac{1}{2}\int_0^t\int_\Omega h(\phi_\delta)\snorm*{\tr{\CC_\delta}\CC_\delta}^2\dx\dta + \frac{1}{2}\int_0^t\int_\Omega h(\phi_\delta)\tr{\CC_\delta}^2\dx\dta \nonumber \\
&+ \left(\int_\Omega \frac{c_0}{2}\snorm{\nabla\phi_\delta(0)}^2 + F_\delta(\phi_\delta(0))+\frac{1}{2}\snorm{q_\delta(0))}^2 +\frac{1}{2}\snorm{\u_\delta(0)}^2 + \frac{1}{4}\snorm*{\CC_\delta(0)}^2 \dx\right). \nonumber
\end{align}

Following \cite{Brunk.} we can pass to the limit $\delta\to 0$ in $E_\delta$, since we have the necessary strong convergences,
cf.~(\ref{eq:convergencetable}).
We apply the fact that for a weakly converging sequence  $\{g_\delta\}_\delta$ in $L^2(0,t;L^2(\Omega))$ we have
\begin{equation}
\norm*{g}_{L^2(0,t;L^2)}\leq \liminf\limits_{\delta\to 0}\norm{g_\delta}_{L^2(0,t;L^2)}. \label{eq:dissipativelimit}
\end{equation}
Therefore we can pass to the limit in the following terms
\begin{equation*}
\sqrt{\tau(\phi_\delta)^{-1}}q_\delta,\; \na q_\delta, \;\sqrt{\eta(\phi_\delta)}\Du_\delta,\; \na\CC_\delta,\; \sqrt{h(\phi_\delta)}\tr{\CC_\delta}\CC_\delta.
\end{equation*}

Further, relabeling
$\hat{\J}_\delta := \sqrt{m_\delta(\phi_\delta)}\na\mu_\delta$ we can apply the lower semi-continuity of the norm (\ref{eq:dissipativelimit})
for $\hat{\J}_\delta$.  This is possible due to (\ref{eq:Jconv}).
The term $h(\phi_\delta)\tr{\CC_\delta}^2$ can be treated due to the strong convergence of $\CC_\delta$, cf. (\ref{eq:convergencetable}). Consequently, we obtain in the limit $\delta\to 0$
\begin{align}
&\left(\int_\Omega \frac{c_0}{2}\snorm{\nabla\phi(t)}^2 + F(\phi(t))+\frac{1}{2}\snorm{q(t)}^2 +\frac{1}{2}\snorm{\u(t)}^2 + \frac{1}{4}\snorm*{\CC(t)}^2 \dx\right)  \\
\leq&  -\int_0^t\int_\Omega\(\snorm*{\Jh} - \snorm*{\na\big(A(\phi)q\big)}\)^2 \dx\dta - \int_0^t\int_\Omega\frac{1}{\tau(\phi)}q^2\dx\dta  \nonumber\\
& - \varepsilon_1\int_0^t\int_\Omega|\nabla q|^2\dx\dta -\int_0^t\int_\Omega\eta(\phi)\snorm*{\Du}^2\dx\dta - \frac{\varepsilon_2}{2}\int_0^t\int_\Omega|\na\CC|^2\dx\dta  \nonumber \\
&- \frac{1}{2}\int_0^t\int_\Omega h(\phi)\snorm*{\tr{\CC}\CC}^2\dx\dta + \frac{1}{2}\int_0^t\int_\Omega h(\phi)\tr{\CC}^2\dx\dta \nonumber \\
&+ \left(\int_\Omega \frac{c_0}{2}\snorm{\nabla\phi_0}^2 + F(\phi_0)+\frac{1}{2}\snorm{q_0}^2 +\frac{1}{2}\snorm{\u_0}^2 + \frac{1}{4}\snorm*{\CC_0}^2 \dx\right), \nonumber
\end{align}
which concludes the proof.


\section{Different degenerate mobilities}
In this section we discuss the case of two different degenerate mobilities, i.e.~$m(\phi)\neq n(\phi)$.
A quite frequently used assumption is $m(\phi)=n(\phi)^2$. Let us approximate $A(\phi)$ by $A_\delta(\phi_\delta)$ suitably which we will specify later.
Recall that for the entropy function $G_\delta$ we have , cf.~(\ref{eq:entropy})
\begin{align*}
\td\left(\int_\Omega G_\delta(\phi_{\delta})\dx\right) &+ \int_\Omega \(\u_\delta\cdot\nabla\phi_\delta\) G_\delta^\prime(\phi_\delta)\dx  + \int_\Omega m_\delta(\phi_\delta)G^{\prime\prime}_\delta(\phi_\delta)\nabla\mu_\delta\cdot\nabla\phi_\delta \dx \\
& = \int_\Omega n_\delta(\phi_\delta)G^{\prime\prime}_\delta(\phi_\delta)\na\big(A_\delta(\phi_\delta)q_\delta\big)\cdot\na\phi_\delta \dx.
\end{align*}
By the construction the entropy function $G_\delta$ satisfies, cf.~(\ref{eq:appent}),
\begin{align*}
m_\delta(\phi_\delta) G^{\prime\prime}_\delta(\phi_\delta) = 1, \qquad n_\delta(\phi_\delta) G^{\prime\prime}_\delta(\phi_\delta) = \frac{1}{n_\delta(\phi_\delta)}.
\end{align*}
Using these properties and the divergence freedom of $\u_\delta$ we observe
\begin{align}
\td\left(\int_\Omega G_\delta(\phi_{\delta})\dx\right) + \int_\Omega \nabla\mu_\delta\cdot\nabla\phi_\delta \dx  = \int_\Omega \frac{1}{n_\delta(\phi_\delta)}\na\big(A_\delta(\phi_\delta)q_\delta\big)\cdot\na\phi_\delta \dx. \label{eq:Gev}
\end{align}
Considering the right hand side of (\ref{eq:Gev}) by expansion of the gradient we find
\begin{align*}
\int_\Omega \frac{1}{n_\delta(\phi_\delta)}\na\big(A_\delta(\phi_\delta)q_\delta\big)\cdot\na\phi_\delta \dx &= \int_\Omega \frac{A_\delta(\phi_\delta)}{n_\delta(\phi_\delta)}\na q_\delta\cdot\na\phi_\delta + \frac{A_\delta^\prime(\phi_\delta)}{n_\delta(\phi_\delta)}q_\delta\snorm*{\na\phi_\delta}^2 \dx \\
&\leq \norm*{\frac{A_\delta(\phi_\delta)}{n_\delta(\phi_\delta)}}_{L^\infty}\norm*{\na q_\delta}_{L^2}\norm*{\na\phi_\delta}_{L^2} + \norm*{\frac{A_\delta^\prime(\phi_\delta)}{n_\delta(\phi_\delta)}}_{L^\infty}\norm*{q_\delta}_{L^2}\norm*{\na\phi_\delta}_{L^4}^2.
\end{align*}

Now, applying the interpolation inequalities and analogous calculations as in Section 6 we obtain
\begin{align}
&\td\left(\int_\Omega G_\delta(\phi_{\delta})\dx\right) + (c_0 -\sigma)\norm*{\Delta\phi_\delta}_{L^2}^2 + \int_\Omega F_{1,\delta}^{\prime\prime}(\phi_\delta)\snorm*{\na\phi_\delta}^2 \dx  \label{eq:curcial}\\
&\leq  c(\sigma)\(\norm*{F_2^{\prime\prime}}_{L^\infty} + \norm*{\frac{A_\delta^\prime(\phi_\delta)}{n_\delta(\phi_\delta)}}_{L^\infty}^2\norm*{q_\delta}_{L^2}^2 \)\norm*{\na\phi_\delta}_{L^2}^2 + \norm*{\frac{A_\delta(\phi_\delta)}{n_\delta(\phi_\delta)}}_{L^\infty}\norm*{\na q_\delta}_{L^2}\norm*{\na\phi_\delta}_{L^2}. \nonumber
\end{align}
Assuming that
\begin{equation}
\norm*{\frac{A_\delta(\phi_\delta)}{n_\delta(\phi_\delta)}}_{L^\infty}\leq c \text{ and } \norm*{\frac{A_\delta^\prime(\phi_\delta)}{n_\delta(\phi_\delta)}}_{L^\infty}\leq c, \label{eq:twodegaass}
\end{equation}
independently of $\delta$, we recover the estimates
\begin{align}
\norm*{\Delta\phi_\delta}_{L^2(L^2)} + \norm*{\nabla\big(A(\phi_\delta)q_\delta\big)}_{L^2(L^2)} + \norm*{\int_\Omega G_\delta(\phi_\delta) \dx}_{L^\infty(0,T)}&\leq c. \label{eq:twodegestimates}
\end{align}
Taking into account (\ref{eq:twodegestimates}) the results of Sections 5-8 can be applied analogously to show the existence of a
weak solution for the case $m(\phi)=n(\phi)^2$. Note that now we have approximated $A$ by $A_\delta$ suitably such that (\ref{eq:twodegaass}) holds.  Indeed, without these additional assumptions on $A_\delta$  and $A$ the right hand side in (\ref{eq:Gev}) would be  unbounded.
The weak solution in the case of different mobilities $m(\phi)\neq n(\phi)$ can be defined in the following way.
\begin{Definition}
	\label{defn:weak_sol_deg_full_full}	
	Let the initial data be given  $$(\phi_0,q_0,\u_0,\CC_0)\in H^1(\Omega)\times L^2(\Omega)\times H\times L^2(\Omega)^{2\times2}.$$
	The quadruple $(\phi,q,\J,\u,\CC)$ is called a weak solution of (\ref{eq:full_model}) if
	\begin{align*}	
	&\phi \in L^{\infty}(0,T;H^1(\Omega))\cap L^2(0,T;H^2(\Omega)),& &q,\CC\in L^\infty(0,T;L^2(\Omega))\cap L^2(0,T;H^1(\Omega)), \\
	&\u\in L^{\infty}(0,T;H)\cap L^2(0,T;V),& &\J=n(\phi)\Jh,\Jh\in L^2(0,T;L^2(\Omega))
	\end{align*}
	and
	\begin{align*}
	&\phi^\prime \equiv \frac{\partial \phi}{\partial t} \in L^{2}(0,T;H^{-1}(\Omega)),&
& q^\prime \equiv \frac{\partial q}{\partial t}, \CC^\prime \equiv \frac{\partial \CC}{\partial t} \in L^{4/3}(0,T;H^{-1}(\Omega)),&
&\u^\prime \equiv \frac{\partial u}{\partial t}\in L^2(0,T;V^{*}).
	\end{align*}
Further, for any test function $(\psi,\zeta,\bxi,\vv,\DD)\in H^1(\Omega)^2\times H^1(\Omega)\cap L^\infty(\Omega)\times V\times H^1(\Omega)^{2\times 2}$ and almost every $t\in(0,T)$ it holds
	\begin{align}
	&\int_\Omega\frac{\partial \phi}{\partial t}\psi \dx + \int_\Omega\(\u\cdot\nabla\phi\)\psi \dx + \int_\Omega \J\cdot\nabla\psi \dx - \int_\Omega n(\phi)\nabla\big(A(\phi)q\big)\cdot\na\psi \dx =0 \nonumber\\
	&\int_\Omega\frac{\partial q}{\partial t}\zeta \dx + \int_\Omega\(\u\cdot\nabla q\)\zeta \dx + \int_\Omega\frac{q\zeta}{\tau(\phi)}\dx +  \int_\Omega\nabla\big(A(\phi)q\big)\cdot\na\big(A(\phi)\zeta\big)\dx + \int_\Omega \varepsilon_1\nabla q\cdot\nabla\zeta\dx    \nonumber\\
	&\hspace{7em} = \int_\Omega \Jh\cdot\nabla\big(A(\phi)\zeta\big) \dx \nonumber\\
	&\int_\Omega \J\cdot\bxi \dx = c_0\int_\Omega \Delta\phi\div{m(\phi)\bxi}\dx + \int_\Omega m(\phi)F^{\prime\prime}(\phi)\nabla\phi\cdot\bxi \dx \\
	&\int_\Omega\frac{\partial\u}{\partial t}\cdot\vv \dx + \int_\Omega(\u\cdot\na)\u\cdot\vv\dx + \int_\Omega\eta(\phi)\Du:\mathrm{D}\vv \dx + \int_\Omega\TT:\nabla\vv \dx + \int_\Omega c_0\Delta\phi\nabla\phi\cdot\vv \dx = 0 \nonumber\\
	&\int_\Omega\frac{\partial\CC}{\partial t}:\DD \dx + \int_\Omega(\u\cdot\na)\CC:\DD\dx  - \int_\Omega\Big[(\nabla\u)\CC + \CC(\na\u)^T\Big]:\DD \dx + \varepsilon_2\int_\Omega\nabla\CC:\nabla\DD \dx & \nonumber\\
	&\hspace{7em}=  - \int_\Omega h(\phi)\tr{\CC}^2\CC:\DD \dx + \int_\Omega h(\phi)\tr{\CC}\I:\DD \dx. \nonumber&
	\end{align}
Furthermore, the initial data $(\phi(0),q(0),\u(0),\CC(0))=(\phi_0,q_0,\u_0,\CC_0)$ are attained.
\end{Definition}

Here the only difference is the appearance of $\J$ and $\Jh$ in the definition of a weak solution. In the energy inequality we recover the full cross-diffusion difference with the term $(\Jh - \na(A(\phi)q))^2$.
\begin{align}
&\left(\int_\Omega \frac{c_0}{2}\snorm{\nabla\phi(t)}^2 + F(\phi(t)) + \frac{1}{2}|q(t)|^2 +\frac{1}{2}\snorm{\u(t)}^2 + \frac{1}{4}\snorm*{\CC(t)}^2 \dx\right) \nonumber \\
&\leq -\int_0^t\int_\Omega \(\Jh - \na\big(A(\phi)q\big)\)^2\dx\dta -\int_0^t\int_\Omega \frac{1}{\tau(\phi)}q^2 \dx\dta -\varepsilon_1\int_0^t\int_\Omega \snorm*{\nabla q}^2 \dx\dta \nonumber \\
&-\int_0^t\int_\Omega \eta(\phi)\snorm*{\Du}^2\dx\dta - \frac{\varepsilon_2}{2}\int_0^t\int_\Omega\snorm*{\nabla\CC}^2\dx\dta - \frac{1}{2}\int_0^t\int_\Omega h(\phi)\snorm*{\trC\CC}^2\dx\dta \label{eq:energy_full_intint}  \\
&+ \frac{1}{2}\int_0^t\int_\Omega h(\phi)\snorm*{\trC}^2\dx\dta + \left(\int_\Omega \frac{c_0}{2}\snorm{\nabla\phi_0}^2 + F(\phi_0) + \frac{1}{2}|q_0|^2 +\frac{1}{2}\snorm{\u_0}^2 + \frac{1}{4}\snorm*{\CC_0}^2 \dx\right). \nonumber
\end{align}

\begin{Remark}
	We want to discuss the consequence of the $L^\infty$ bounds in Assumption~\ref{ass:deg2}, cf.~\eqref{eq:twodegaass}. We  write $n(s)$ as $n(s)=s^\beta(1-s)^\beta N(s)$ for some $\beta \geq 1$ and a bounded positive smooth function $N$. Since $n^{-1}(s)$ is unbounded for $s\in\{0,1\}$ we need that $A,A^\prime$ are vanishing at $\{0,1\}$ with some rate $\alpha\geq\beta$. This implies that locally around $\{0,1\}$ we can set $A(s)=s^\alpha(1-s)^\alpha A_{b}(s)$, for a bounded positive smooth function $A_b$.
	\begin{align}
	\frac{A(s)}{n(s)}&= \Big(s(1-s)\Big)^{\alpha-\beta}\frac{A_{b}(s)}{N(s)}, \nonumber\\
	 \frac{A^\prime(s)}{n(s)}&= \alpha\Big(s(1-s)\Big)^{\alpha-1-\beta}(1-2s)\frac{A_{b}(s)}{N(s)} + \Big(s(1-s)\Big)^{\alpha-\beta}\frac{A^\prime_{b}(s)}{N(s)}. \label{eq:Ablowup}
	\end{align}
	This implies that $c_1s^{\beta+1}\leq A(s)\leq c_2s^{\beta+1}$ as $s\to 0$ or $s\to 1$.
\end{Remark}
From the physical point of view there is no difference to consider the above modification of the function $A$. If the volume fraction fulfills $\phi=1$, then we have in a mixture only pure polymer without any solvent. In this case the evolution of $\phi$ reduces to the transport along the streamlines.  Therefore after expanding the gradient term $n(\phi)\na(A(\phi)q)$, this should go to zero for arbitrary values of $q$. Numerical simulations confirm that the numerical results for a sufficiently small and smooth cut-off of $A(\phi)$ coincide with the simulations for an original $A$.

\section{Numerical Simulations}
In this section we illustrate the behavior of our viscoelastic phase separation model to three-dimensional flows.
{\cred As already mentioned before the viscoelastic phase separation is a complex dynamical process describing phase-separation of a dynamically asymmetric mixture, which is composed of fast and slow phases. This leads to structure formation phenomena, such as transient formation of network-like structures of a polymer-rich phase and its volume shrinking. As we observe below numerical experiments confirm these rich dynamical processes.}

A numerical scheme is based on the Lagrange-Galerkin finite element method from \cite{LukacovaMedvidova.2017b,LukacovaMedvidova.2017c}, see also \cite{Brunk.}. We decompose the computational domain $\Omega=[0,24]^3$ into tetrahedrons. The numerical solution is based on the first order piecewise polynomial approximation. In time we adopted the characteristic scheme to approximate the material derivative.
The following experiment is similar to Experiment~2 from \cite{Brunk.} for the Flory-Huggins potential with a degenerate mobility. The following functions and parameters will be used
\begin{align*}
&m(\phi)=\phi^2(1-\phi)^2, n(\phi)=\phi(1-\phi), A(\phi)=\frac{1}{2}\tanh\(10^{3}\cdot\left[\cot(\pi\phi^*)- \cot(\pi\phi)\right] \),\\
&\tau(\phi)=(5\phi^2)^{-1}, h(\phi)=(5\phi^2)^{-1}, \eta(\phi) = 2+\phi^2, c_0=1, \varepsilon=0.1, \phi^* = 0.4, \\
&n_p=n_s=1, \chi=28/11.
\end{align*}
\textbf{Experiment 1:} In this test we set the initial data to $\phi_0(x) = 0.4 + \xi(x), q_0=0, \u_0=\mathbf{0}, \CC_0=\frac{1}{\sqrt{3}}\I$, where $\xi(x)$ is a random perturbation from $[-10^{-3},10^{-3}]$.\\[0.5em]
Figures \ref{fig:exp1phii}-\ref{fig:exp1en} present the time evolution of numerical solutions $\phi,q,p,\mu$ and $\snorm*{\u}_2$. Similarly as in \cite{Brunk.} we can recognize the frozen phase $(t \approx 20)$, the elastic regime with solvent-rich droplets $(t \approx 130)$, the volume shrinking phase $(t\approx 200)$ and network pattern $(t\approx 500)$. The last regime requires much longer simulations and is therefore not presented here. \\
We present isosurfaces of $\phi$ and $\snorm*{\u}_2$. For the other variables we plot three chosen cuts. Comparing them with the results of two-dimensional simulations presented in \cite{Brunk.} we can see that they match quite well in terms of structure. {\cred Two-dimensional cuts shown in Figure~\ref{fig_exp1phis} are in good agreement with physical experiments, presented as two-dimensional snapshots in \cite{Tanaka.}.  We can clearly recognize the phase inversion followed by the formation of network-structures. During the phase inversion solvent-rich droplets
expand very fast leading to a change in the observable dominant phase. In our simulations phase inversion happens between $t \approx 80$ and $t\approx 200$, see Figure~\ref{fig_exp1phis}.}

The speedup in the time evolution in comparison with two dimensional simulations can be explained in the following way. First, we consider a much smaller domain without rescaling $c_0$, which is directly connected to the timescale. Second, we have rescaled $A(\phi)$ to a more physically reasonable case $A(\phi)\in[0,1]$ instead of $A(\phi)\in[1,2]$, cf.~\cite{Brunk.}. Finally, we can observe from Figure \ref{fig:exp1en} that the scheme is practically energy-stable and mass conservative.
\begin{figure}[H]
	\centering
	\begin{tabular}{cccc}
		\hspace{-0.7cm}\subfloat[][]{\includegraphics[trim={0.8cm 0.4cm 0.8cm 0.2cm},clip,scale=0.45]{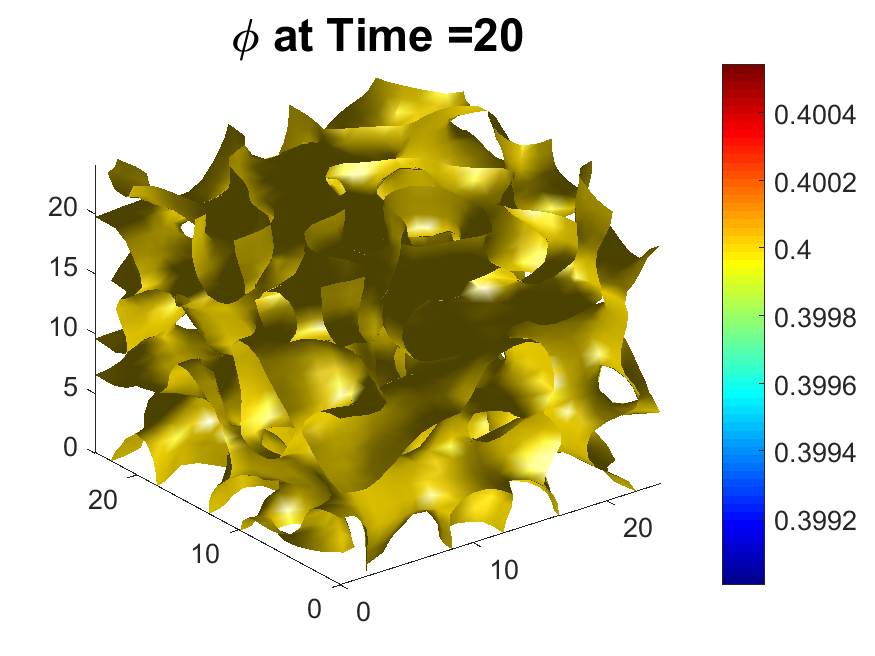}} &
		\subfloat[][]{\includegraphics[trim={0.8cm 0.4cm 0.8cm 0.2cm},clip,scale=0.45]{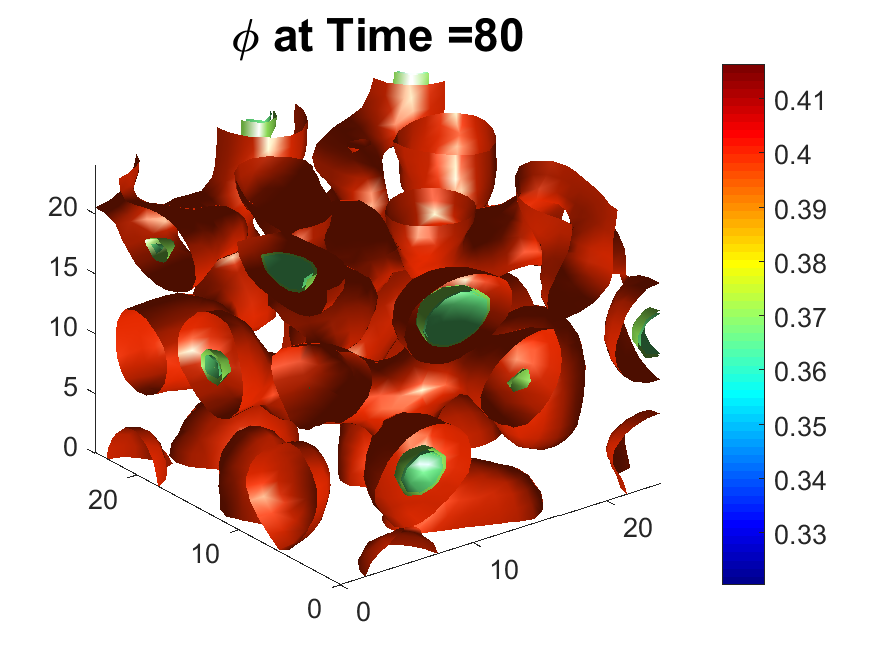}} &
		\subfloat[][]{\includegraphics[trim={0.8cm 0.4cm 0.8cm 0.2cm},clip,scale=0.45]{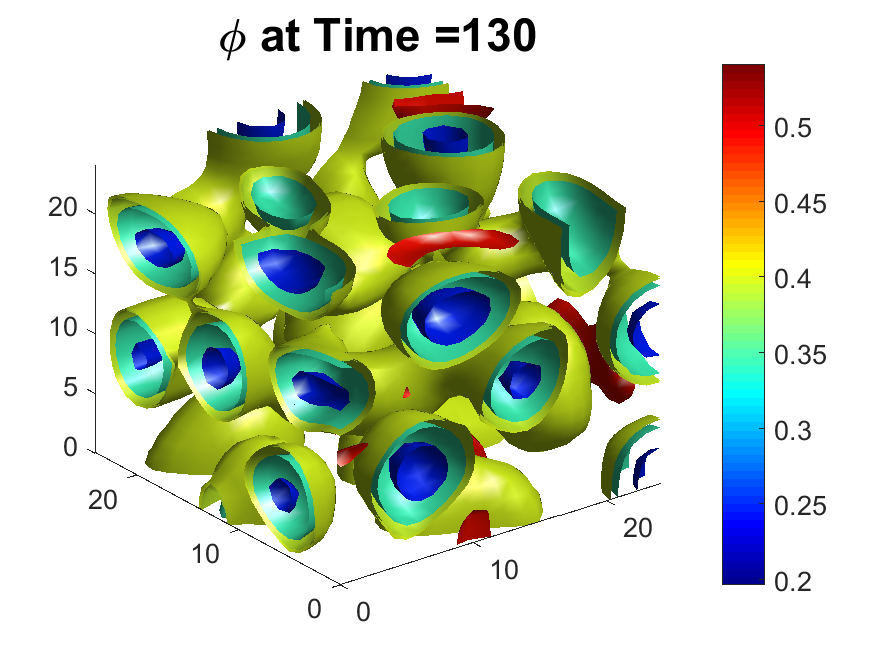}} \\
		\hspace{-0.7cm}\subfloat[][]{\includegraphics[trim={0.8cm 0.4cm 0.8cm 0.2cm},clip,scale=0.45]{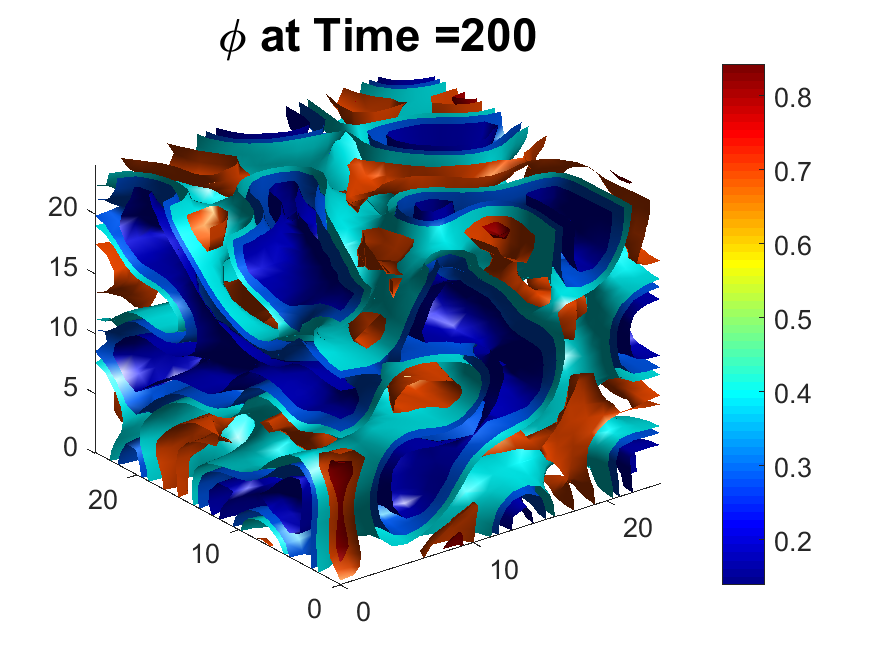}} &
		\subfloat[][]{\includegraphics[trim={0.8cm 0.4cm 0.8cm 0.2cm},clip,scale=0.45]{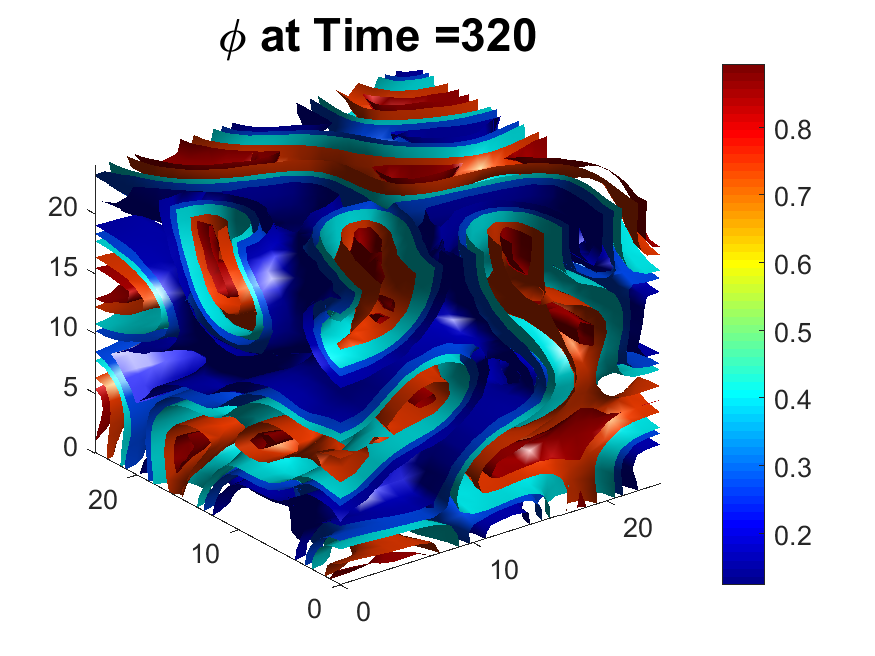}} &
		\subfloat[][]{\includegraphics[trim={0.8cm 0.4cm 0.8cm 0.2cm},clip,scale=0.45]{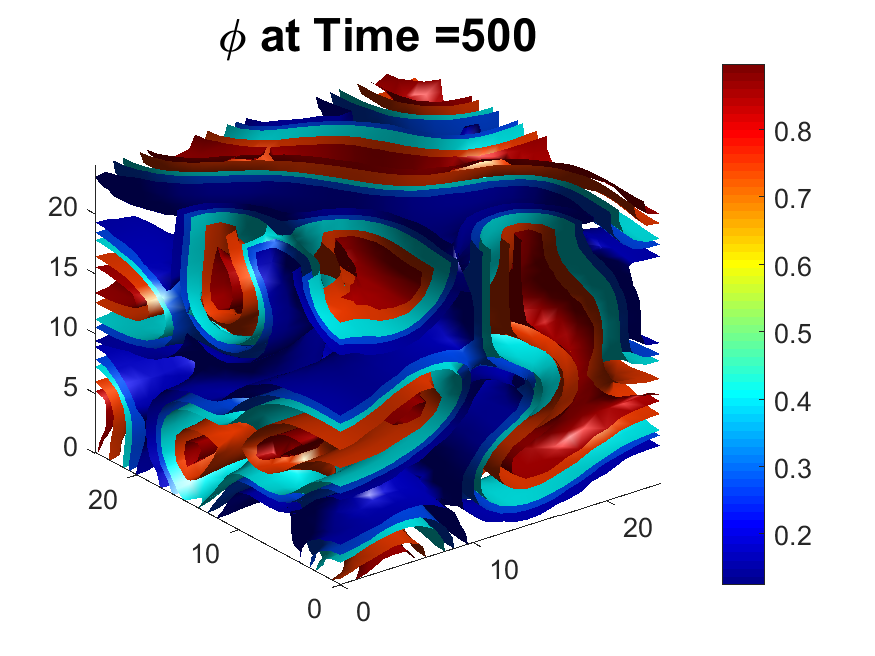}}
	\end{tabular}
	\caption{Isosurface: Spinodal decomposition, time evolution of the volume fraction $\phi$.}
	\label{fig:exp1phii}
\end{figure}

\begin{figure}[H]
	\centering
	\begin{tabular}{ccc}
		\hspace{-0.7cm}\subfloat[][]{\includegraphics[trim={0.8cm 0.4cm 0.3cm 0.2cm},clip,scale=0.43]{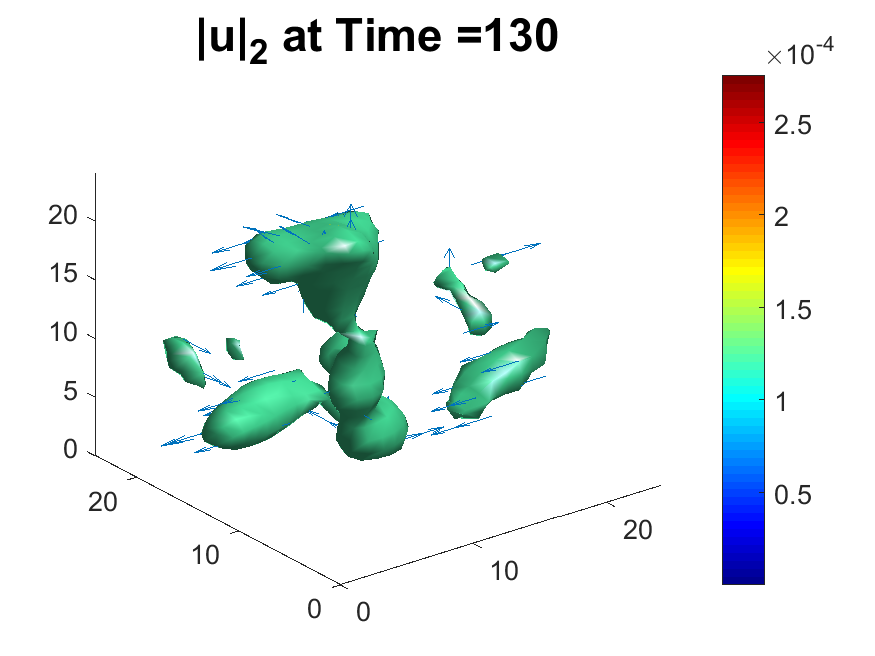}} &
		\subfloat[][]{\includegraphics[trim={0.8cm 0.4cm 0.3cm 0.2cm},clip,scale=0.43]{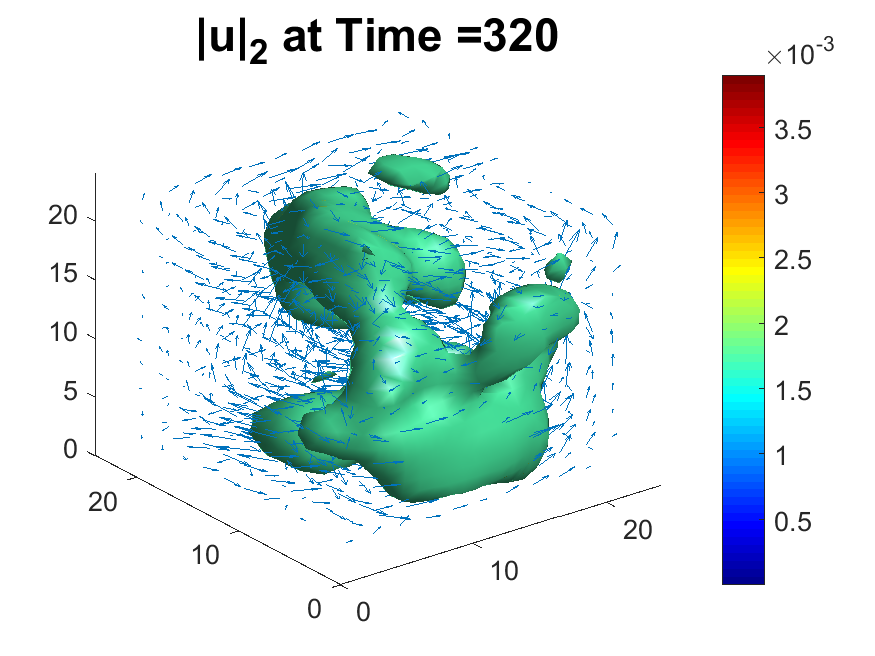}} &
		\subfloat[][]{\includegraphics[trim={0.8cm 0.4cm 0.3cm 0.2cm},clip,scale=0.43]{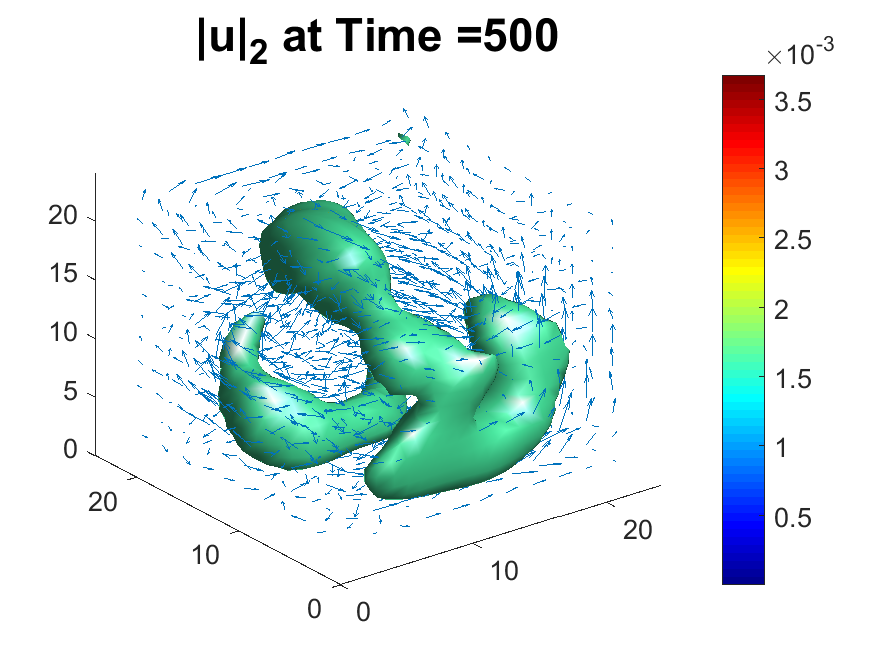}}\\
	\end{tabular}
	\caption{Isosurface: Spinodal decomposition, time evolution of the velocity norm $|\u|_2$ with the vector field.}
	\label{fig_exp1magui}
\end{figure}


\begin{figure}[H]
	\centering
	\begin{tabular}{cccc}
		\hspace{-0.7cm}\subfloat[][]{\includegraphics[trim={0.8cm 0.5cm 0.8cm 0.4cm},clip,scale=0.43]{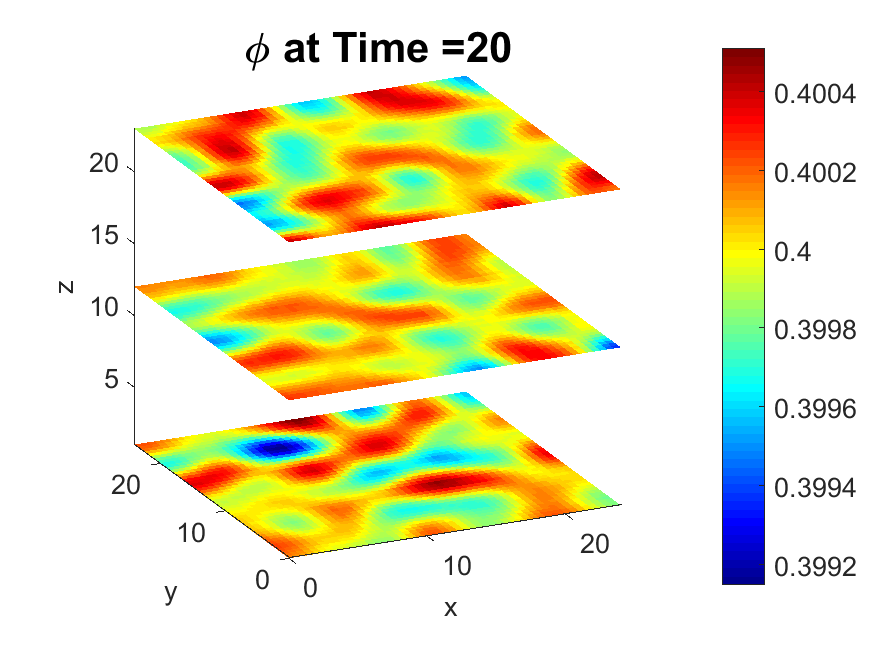}} &
		\subfloat[][]{\includegraphics[trim={0.8cm 0.5cm 0.8cm 0.4cm},clip,scale=0.43]{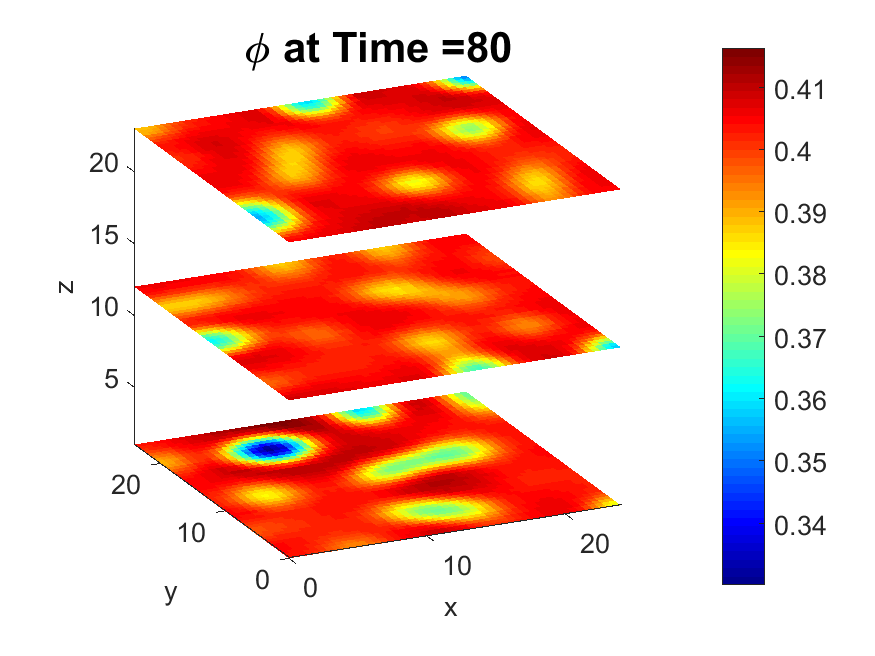}} &
		\subfloat[][]{\includegraphics[trim={0.8cm 0.5cm 0.8cm 0.4cm},clip,scale=0.43]{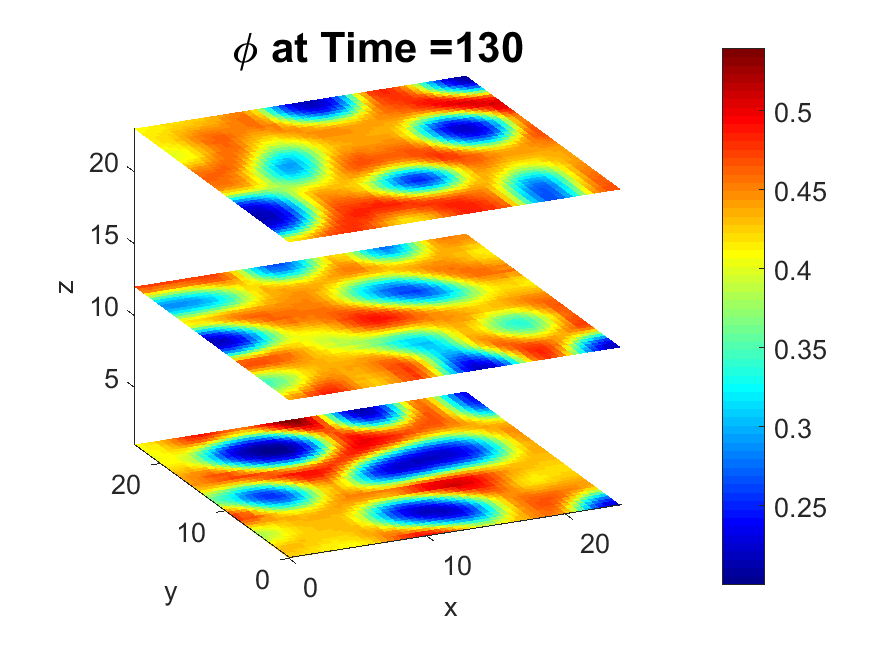}} \\
		\hspace{-0.7cm}\subfloat[][]{\includegraphics[trim={0.8cm 0.5cm 0.8cm 0.4cm},clip,scale=0.43]{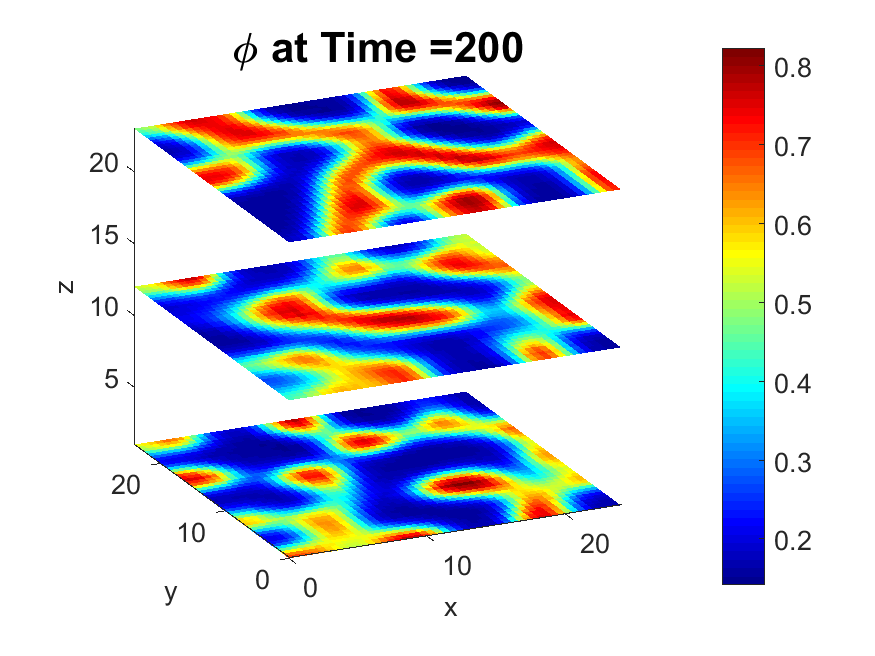}} &
		\subfloat[][]{\includegraphics[trim={0.8cm 0.5cm 0.8cm 0.4cm},clip,scale=0.43]{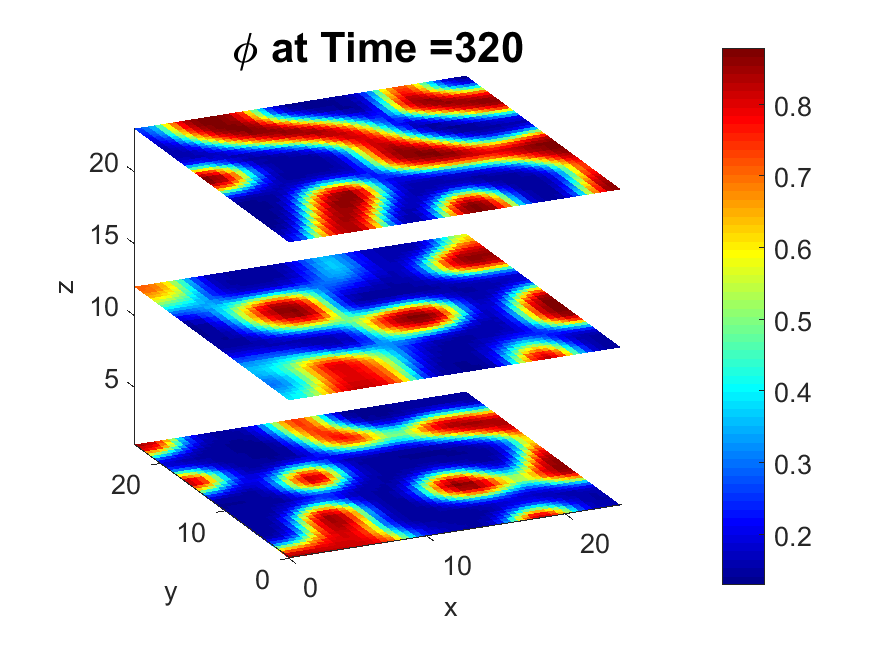}} &
		\subfloat[][]{\includegraphics[trim={0.8cm 0.5cm 0.8cm 0.4cm},clip,scale=0.43]{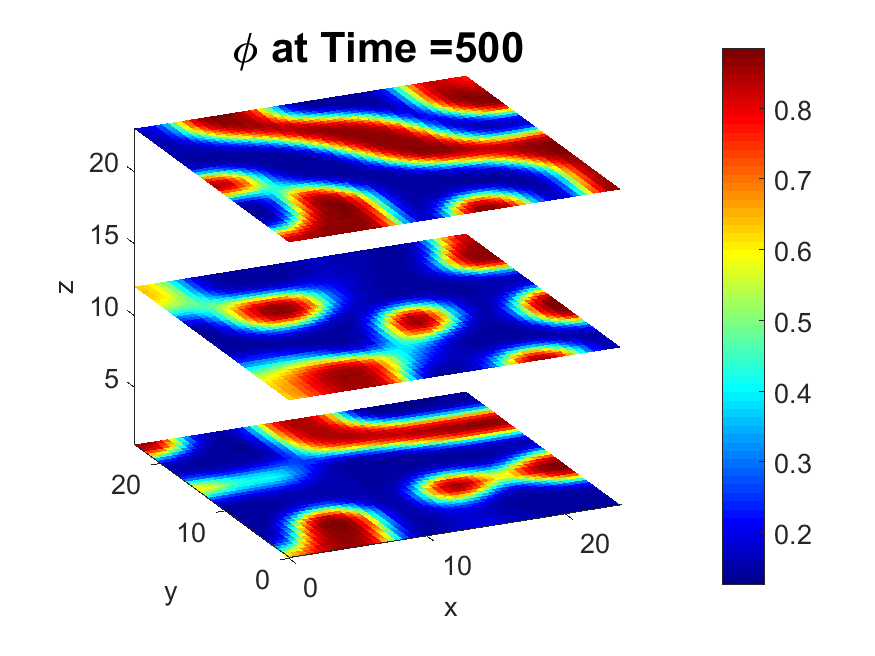}}
	\end{tabular}
	\caption{Slice: Spinodal decomposition, time evolution of the volume fraction $\phi$.}
	\label{fig_exp1phis}
\end{figure}

\begin{figure}[ht]
	\centering
	\begin{tabular}{ccc}
		\hspace{-0.7cm}\subfloat[][]{\includegraphics[trim={0.7cm 0.5cm 0.6cm 0.00cm},clip,scale=0.43]{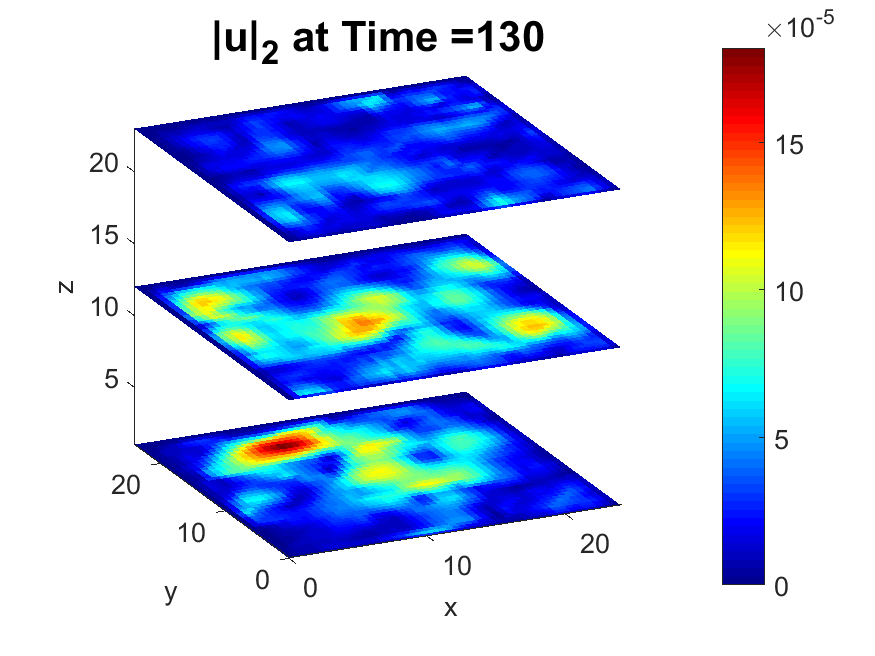}} &
	    \subfloat[][]{\includegraphics[trim={0.7cm 0.5cm 0.6cm 0.00cm},clip,scale=0.43]{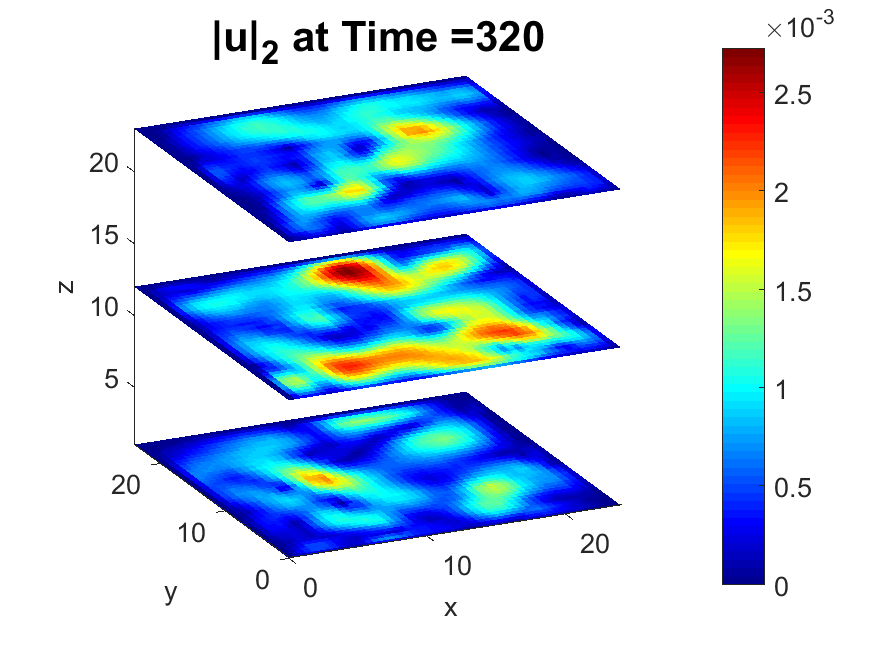}} &
		\subfloat[][]{\includegraphics[trim={0.7cm 0.5cm 0.6cm 0.00cm},clip,scale=0.43]{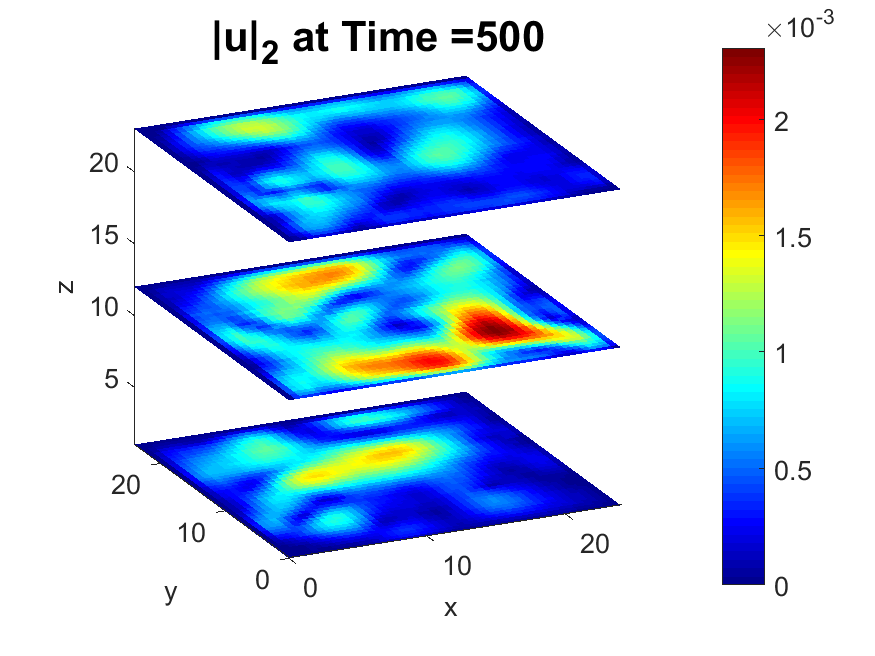}}\\
	\end{tabular}
	\caption{Slice: Spinodal decomposition, time evolution of the velocity norm $|\u|_2$.}
	\label{fig_exp1magus}
\end{figure}

\begin{figure}[h!]
	\centering
	\begin{tabular}{ccc}	
	    \hspace{-0.7cm}\subfloat[][]{\includegraphics[trim={0.7cm 0.5cm 0.3cm 0.00cm},clip,scale=0.43]{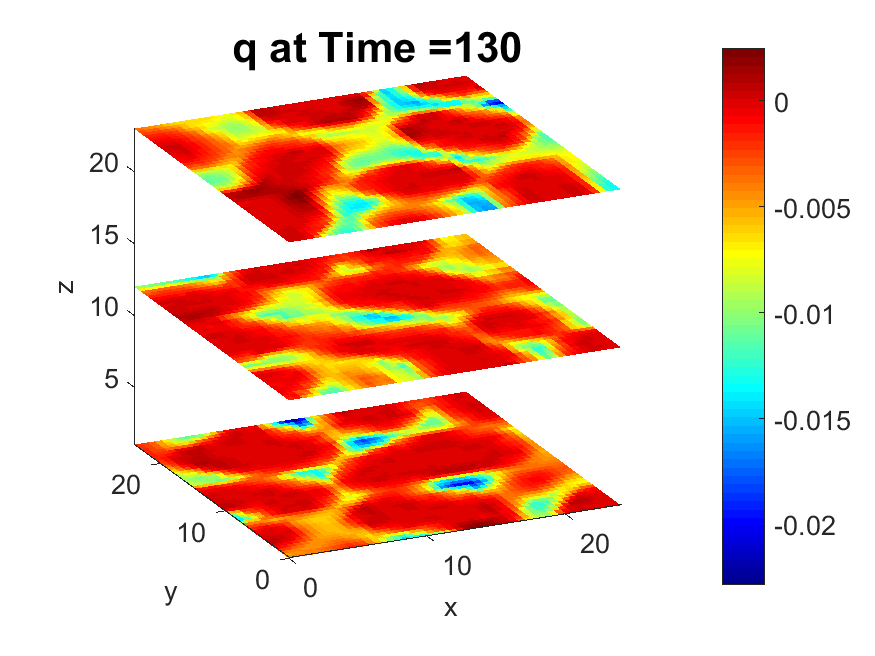}} &
		\subfloat[][]{\includegraphics[trim={0.7cm 0.5cm 0.3cm 0.00cm},clip,scale=0.43]{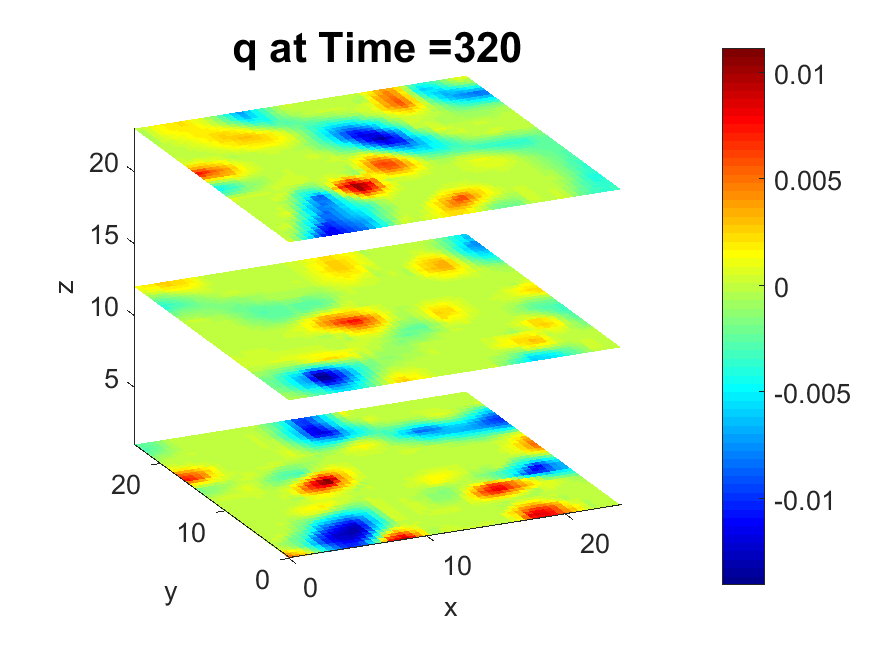}} &
		\subfloat[][]{\includegraphics[trim={0.7cm 0.5cm 0.3cm 0.00cm},clip,scale=0.43]{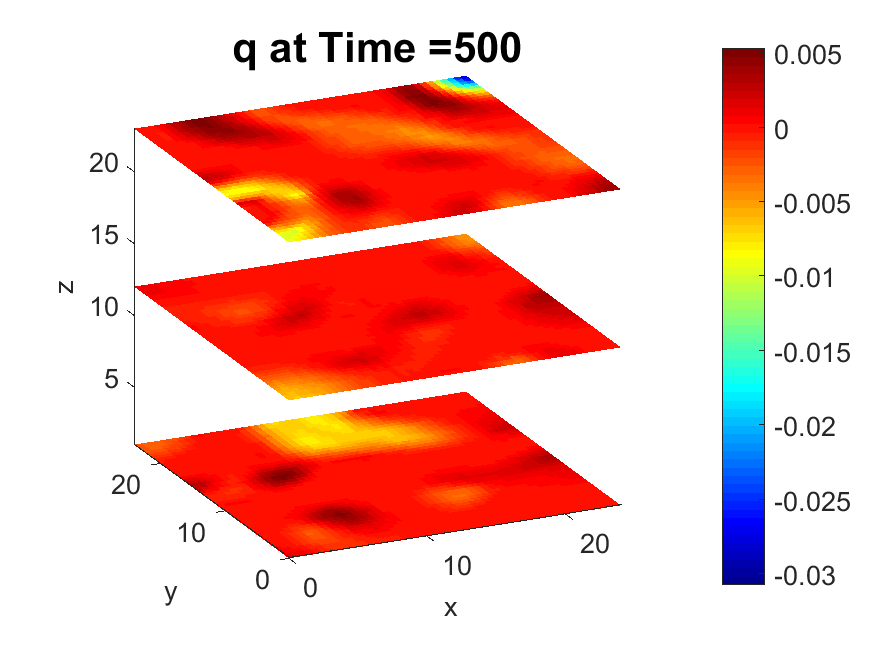}} \\
		\hspace{-0.7cm}\subfloat[][]{\includegraphics[trim={0.7cm 0.5cm 0.3cm 0.00cm},clip,scale=0.43]{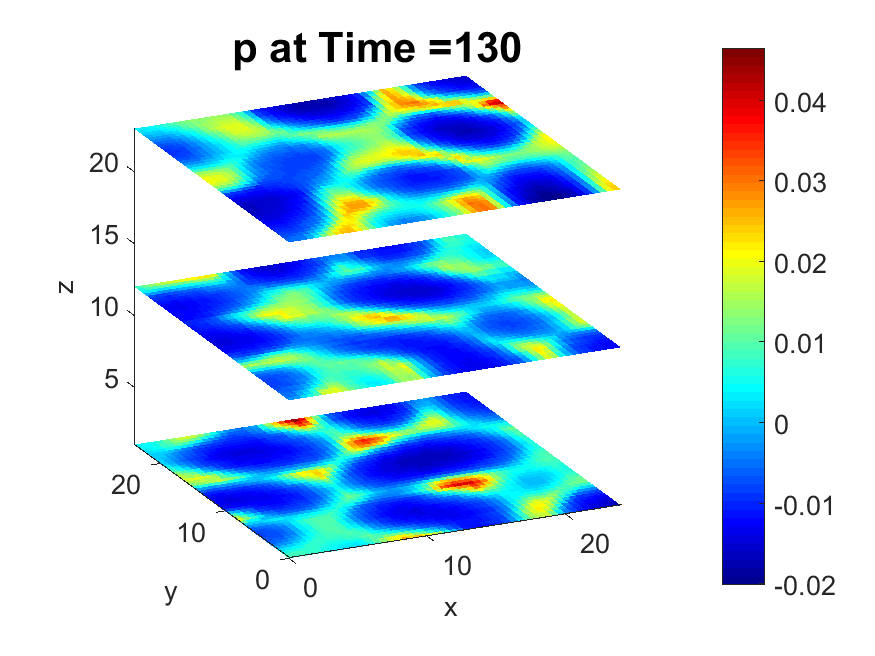}} &
		\subfloat[][]{\includegraphics[trim={0.7cm 0.5cm 0.3cm 0.00cm},clip,scale=0.42]{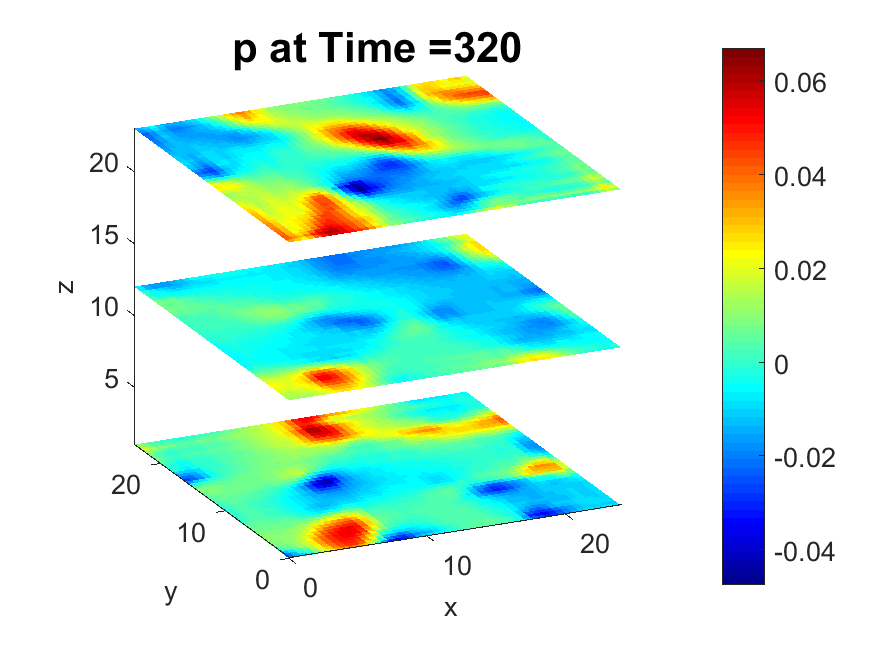}} &
		\subfloat[][]{\includegraphics[trim={0.7cm 0.5cm 0.3cm 0.00cm},clip,scale=0.43]{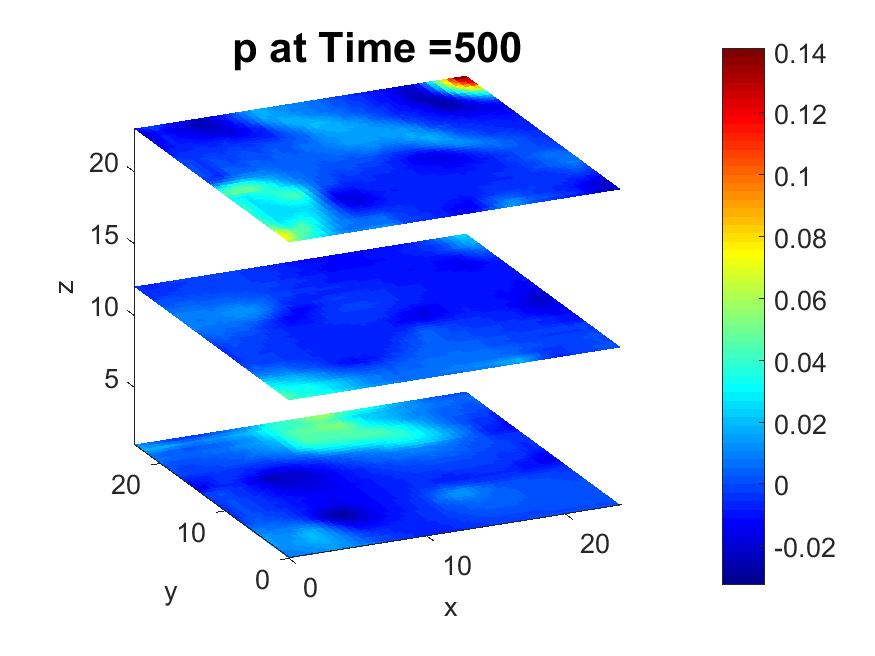}} \\	
		\hspace{-0.7cm}\subfloat[][]{\includegraphics[trim={0.7cm 0.5cm 0.3cm 0.00cm},clip,scale=0.43]{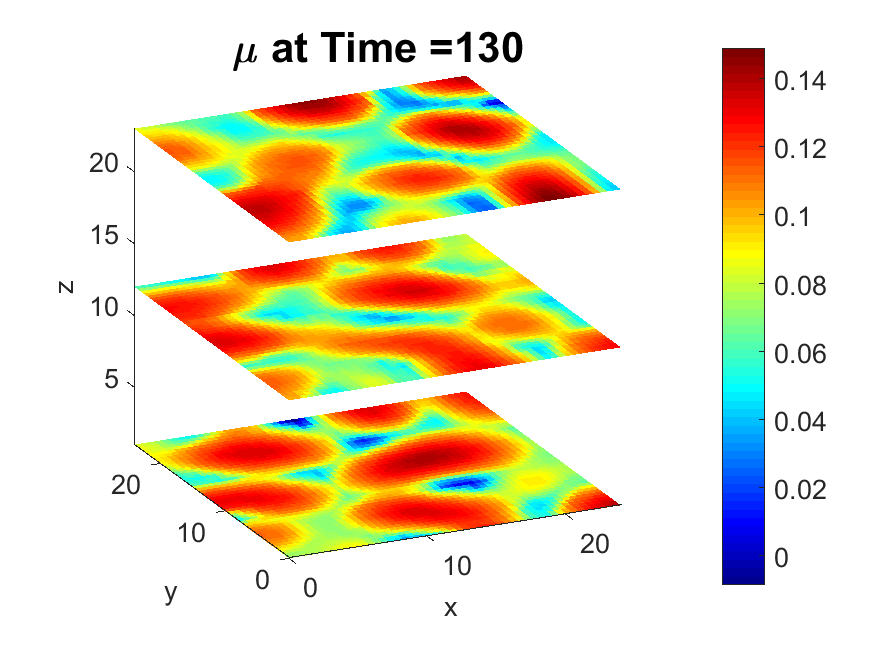}} &
		\subfloat[][]{\includegraphics[trim={0.7cm 0.5cm 0.3cm 0.00cm},clip,scale=0.43]{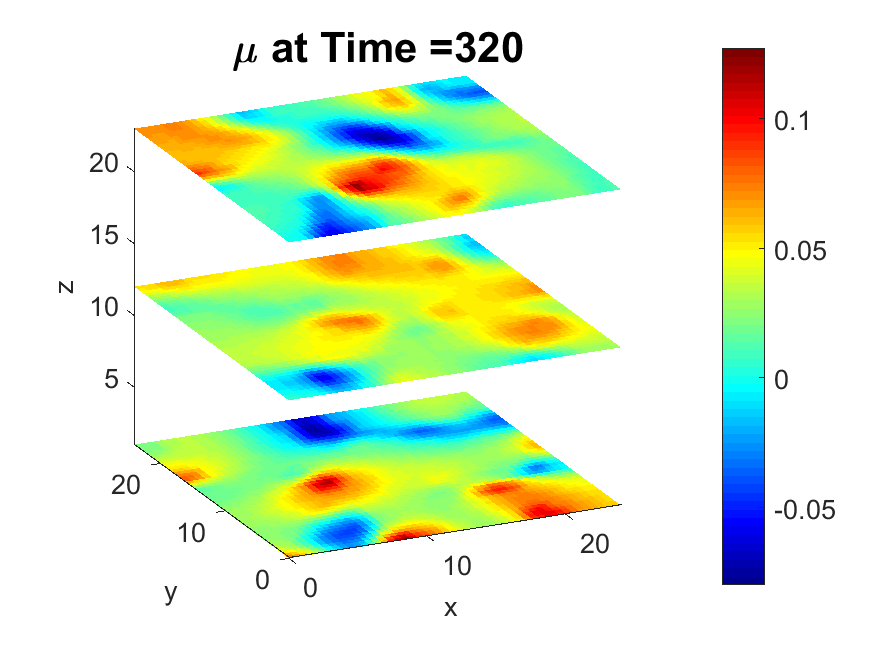}} &
		\subfloat[][]{\includegraphics[trim={0.7cm 0.5cm 0.3cm 0.00cm},clip,scale=0.43]{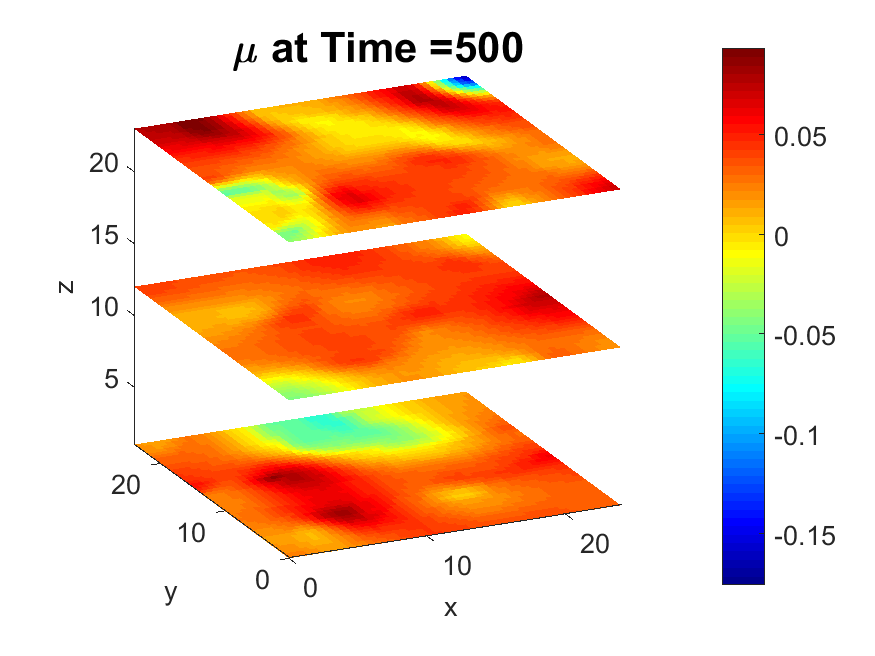}}
	\end{tabular}
	\caption{Slice: Spinodal decomposition, time evolution of the bulk stress $q$ (top), pressure $p$ (middle) and chemical potential $\mu$ (bottom).}
	\label{fig_exp1qpmus}
\end{figure}


\begin{figure}[H]
	\centering
	\includegraphics[trim={5.6cm 0.0cm 0.0cm 0.0cm},clip,scale=0.36]{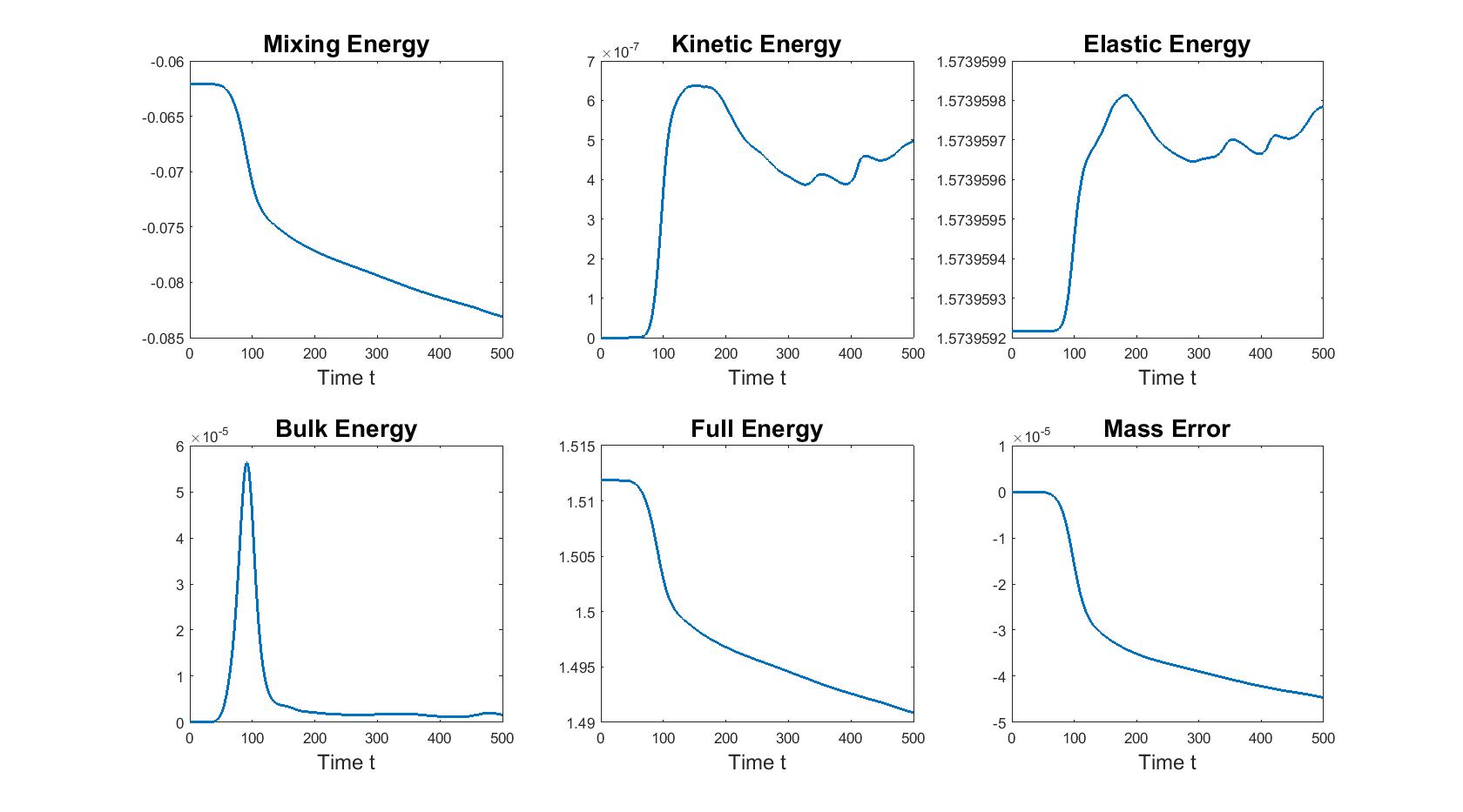}
	\caption{Time evolution of the total energy $E_{tot}$ and the corresponding energy components. The last picture demonstrates that the numerical scheme preserves mass $\frac{1}{\snorm*{\Omega}}\int \phi$ up to small error $\frac{1}{\snorm*{\Omega}}\int \(\phi(x,0) - \phi(x,t)\) \dx$ of the order $10^{-5}$.}
	\label{fig:exp1en}
\end{figure}

\section{Conclusion}
In this paper we have proved the existence of a global in time weak solution to the viscoelastic phase separation model \refeq{eq:full_model},
cf.~Theorem \ref{theo:deg1}. We have extended the approach of \cite{Abels.2013,Boyer.1999,Elliott.1996} for the degenerate mobilities in the Cahn-Hilliard framework to a strongly coupled nonlinear cross-diffusive Cahn-Hilliard system arising in our viscoelastic phase separation model. The crucial difference between the models studied in literature and our model is an additional equation for the bulk stress. Moreover, the coupling terms can degenerate. Using Assumptions \ref{ass:deg} or \ref{ass:deg2} on the nonlinear parameter function $A(\phi)$,
cf.~\eqref{eq:twodegaass}, we were able to extend the approach of \cite{Abels.2013,Boyer.1999,Elliott.1996} to a more complex model with a cross-diffusion coupling. The behavior of the viscoelastic phase separation model was illustrated in Section 10.
As far as we know this is the first result in literature where the viscoelastic phase separation with degenerate mobilities and a singular potential has been analysed rigorously.

\section*{Acknowledgment}
This research was supported by the German Science Foundation (DFG) under the Collaborative Research Center TRR~146 Multiscale Simulation Methods for Soft Matters (Project~C3) and partially by the
VEGA grant~1/0684/17. We would like to thank B.~D\"unweg, D.~Spiller and J.~Ka{\softt}uchov\'a for fruitful discussions on the topic.

\bibliography{paper_bib}

\begin{thebibliography}{10}

\bibitem{Abels.2013}
H.~Abels, D.~Depner, and H.~Garcke.
\newblock {On an incompressible Navier--Stokes/Cahn--Hilliard system with
  degenerate mobility}.
\newblock {\em {Ann I H Poincare-An}}, 30(6):1175--1190, 2013.

\bibitem{Abels.2008}
H.~Abels and E.~Feireisl.
\newblock {On a diffuse interface model for a two-phase flow of compressible
  viscous fluids}.
\newblock {\em {Indiana U Math J}}, pages 659--698, 2008.

\bibitem{ABELS2012}
H.~Abels, H.~Garcke, and G.~Gr\"{u}n.
\newblock {Thermodynamically} {consistent}, {frame} {indifferent} {diffusive}
  {interface} {models} {FOR} {incompressible} {two}-{phase} {flows} {with}
  {different} {denstities}.
\newblock {\em Math Models Methods Appl Sci}, 22(03):1150013, 2012.

\bibitem{agosti_cahn-hilliard-type_2017}
A.~Agosti, P.~F. Antonietti, P.~Ciarletta, M.~Grasselli, and M.~Verani.
\newblock A {Cahn}-{Hilliard}-type equation with application to tumor growth
  dynamics.
\newblock {\em Math Methods Appl Sci}, 40(18):7598--7626, 2017.

\bibitem{Barrett2007}
J.~W. Barrett and E.~S\"{u}li.
\newblock Existence of global weak solutions to some regularized kinetic models
  for dilute polymers.
\newblock {\em Multiscale Model Sim}, 6(2):506--546, 2007.

\bibitem{Boyer.1999}
F.~Boyer.
\newblock {Mathematical study of multiphase flow under shear through order
  parameter formulation}.
\newblock {\em {Asymptotic Anal}}, pages 175--212, 1999.

\bibitem{brunk2020analysis}
A.~Brunk, B.~Dünweg, H.~Egger, O.~Habrich, M.~Lukacova-Medvidova, and
  D.~Spiller.
\newblock Analysis of a viscoelastic phase separation model, 2021.
\newblock accepted to JPCM, doi: 10.1088/1361-648X/abeb13.

\bibitem{Brunk.}
A.~Brunk and M.~Luk{\'a}{\v{c}}ov{\'a}-Medvi{\softd}ov{\'a}.
\newblock {Global existence of weak solutions to viscoelastic phase separation:
  Part I Regular Case}.
\newblock Submitted to Nonlinearity, 2019.

\bibitem{cances_two-phase_2019}
C.~Cancès, D.~Matthes, and F.~Nabet.
\newblock A {Two}-{Phase} {Two}-{Fluxes} {Degenerate} {Cahn}–{Hilliard}
  {Model} as {Constrained} {Wasserstein} {Gradient} {Flow}.
\newblock {\em Arch Ratio Mech Anal}, 233(2):837--866, 2019.

\bibitem{dai_weak_2021}
S.~Dai, Q.~Liu, and K.~Promislow.
\newblock Weak solutions for the functionalized {Cahn}–{Hilliard} equation
  with degenerate mobility.
\newblock {\em Appl Anal}, 100(1):1--16, 2021.

\bibitem{Elliott.1996}
C.~M. Elliott and H.~Garcke.
\newblock {On the Cahn--Hilliard equation with degenerate mobility}.
\newblock {\em {SIAM J Math Anal}}, 27(2):404--423, 1996.

\bibitem{Folland.2011}
G.~B. Folland.
\newblock {\em {Real Analysis: Modern Techniques and their Applications}}.
\newblock {Pure A Math}. {John Wiley {\&} Sons}, New York, 2 edition, 2011.

\bibitem{Gilbarg.1977}
D.~Gilbarg and N.~S. Trudinger.
\newblock {\em {Elliptic Partial Differential Equations of Second Order}},
  volume 224 of {\em {Grundlehren der mathematischen Wissenschaften}}.
\newblock Springer, 1977.

\bibitem{Grun.1995}
G.~Gr{\"u}n.
\newblock {Degenerate parabolic differential equations of fourth order and a
  plasticity model with non-local hardening}.
\newblock {\em {Z Anal Anwend}}, 14(3):541--574, 1995.

\bibitem{Grn2016}
G.~Gr\"{u}n and S.~Metzger.
\newblock On micro{\textendash}macro-models for two-phase flow with dilute
  polymeric solutions {\textemdash} modeling and analysis.
\newblock {\em Math Models Methods Appl Sci}, 26(05):823--866, 2016.

\bibitem{Grn2017}
G.~Gr\"{u}n and S.~Metzger.
\newblock Micro-macro-models for two-phase flow of dilute polymeric solutions:
  Macroscopic limit, analysis, and numerics.
\newblock In {\em Transport Processes at Fluidic Interfaces}, pages 291--303.
  Springer International Publishing, 2017.

\bibitem{Hohenberg.1977}
P.~C. Hohenberg and B.~I. Halperin.
\newblock {Theory of dynamic critical phenomena}.
\newblock {\em {Rev Mod Phys}}, 49(3):435--479, 1977.

\bibitem{jihui_degenerate_2020}
W.~Jihui and W.~Shu.
\newblock On the degenerate {Cahn}–{Hilliard} equation: {Global} existence
  and entropy estimates of weak solutions.
\newblock {\em Asymptotic Anal}, 119(1-2):1--38, 2020.

\bibitem{liu_convective_2008}
C.~Liu.
\newblock On the convective {Cahn}–{Hilliard} equation with degenerate
  mobility.
\newblock {\em J Math Anal Appl}, 344(1):124--144, 2008.

\bibitem{Lowengrub.1998}
J.~Lowengrub and L.~Truskinovsky.
\newblock {Quasi--incompressible Cahn--Hilliard fluids and topological
  transitions}.
\newblock {\em {P Roy Soc A-Math Phy}}, 454(1978):2617--2654, 1998.

\bibitem{LukacovaMedvidova.2017}
M.~Luk{\'a}{\v{c}}ov{\'a}-Medvi{\softd}ov{\'a}, B.~D{\"u}nweg, P.~Strasser, and
  N.~Tretyakov.
\newblock {Energy-stable numerical schemes for multiscale simulations of
  polymer--solvent mixtures}.
\newblock In P.~{van Meurs}, M.~Kimura, and H.~Notsu, editors, {\em
  {Mathematical Analysis of Continuum Mechanics and Industrial Applications
  II}}, volume~30 of {\em {Math Ind}}, pages 153--165. Springer, Singapore,
  2017.

\bibitem{LukacovaMedvidova.2015}
M.~Luk{\'a}{\v{c}}ov{\'a}-Medvi{\softd}ov{\'a}, H.~Mizerov{\'a}, and
  {\v{S}}.~Ne{\v{c}}asov{\'a}.
\newblock {Global existence and uniqueness result for the diffusive Peterlin
  viscoelastic model}.
\newblock {\em {Nonlinear Anal-Theor}}, 120:154--170, 2015.

\bibitem{LukacovaMedvidova.2017d}
M.~Luk{\'a}{\v{c}}ov{\'a}-Medvi{\softd}ov{\'a}, H.~Mizerov{\'a},
  {\v{S}}.~Ne{\v{c}}asov{\'a}, and M.~Renardy.
\newblock {Global existence result for the generalized Peterlin viscoelastic
  model}.
\newblock {\em {SIAM J Math Anal}}, 49(4):2950--2964, 2017.

\bibitem{LukacovaMedvidova.2017b}
M.~Luk{\'a}{\v{c}}ov{\'a}-Medvi{\softd}ov{\'a}, H.~Mizerov{\'a}, H.~Notsu, and
  M.~Tabata.
\newblock {Numerical analysis of the Oseen-type Peterlin viscoelastic model by
  the stabilized Lagrange--Galerkin method. Part I: A nonlinear scheme}.
\newblock {\em {ESAIM: M2AN}}, 51(5):1637--1661, 2017.

\bibitem{LukacovaMedvidova.2017c}
M.~Luk{\'a}{\v{c}}ov{\'a}-Medvi{\softd}ov{\'a}, H.~Mizerov{\'a}, H.~Notsu, and
  M.~Tabata.
\newblock {Numerical analysis of the Oseen-type Peterlin viscoelastic model by
  the stabilized Lagrange--Galerkin method. Part II: A linear scheme}.
\newblock {\em {ESAIM: M2AN}}, 51(5):1663--1689, 2017.

\bibitem{metzger2018}
S.~Metzger.
\newblock On convergent schemes for two-phase flow of dilute polymeric
  solutions.
\newblock {\em ESAIM: M2AN}, 52:2357--2408, 2018.

\bibitem{Mizerova.2015}
H.~Mizerov\'a.
\newblock {\em {Analysis and numerical solution of the Peterlin viscoelastic
  model}}.
\newblock {Dissertation}, {Johannes Gutenberg-Universit{\"a}t}, Mainz, 2015.

\bibitem{Passo1998}
R.~D. Passo, H.~Garcke, and G.~Gr\"{u}n.
\newblock On a fourth-order degenerate parabolic equation: Global entropy
  estimates, existence, and qualitative behavior of solutions.
\newblock {\em {SIAM} J Math Anal}, 29(2):321--342, 1998.

\bibitem{Strasser.2018}
P.~J. Strasser, G.~Tierra, B.~D{\"u}nweg, and
  M.~Luk{\'a}{\v{c}}ov{\'a}-Medvi{\softd}ov{\'a}.
\newblock {Energy-stable linear schemes for polymer--solvent phase field
  models}.
\newblock {\em {Comput Math Appl}}, 2018.

\bibitem{Tanaka.}
H.~Tanaka.
\newblock {Viscoelastic phase separation}.
\newblock {\em {J. Phys.: Condens. Matter}}, 12(15):R207, 2000.

\bibitem{Zhou.2006}
D.~Zhou, P.~Zhang, and W.~E.
\newblock {Modified models of polymer phase separation}.
\newblock {\em {Phys Rev E}}, 73(6):061801, 2006.

\end{thebibliography}
\bibliographystyle{abbrv}

\end{document}